\documentclass[11pt,reqno,a4paper]{amsart}

\usepackage{etoolbox}
\usepackage{amsmath}
\usepackage{enumerate}
\usepackage{amssymb}
\usepackage{amscd}
\usepackage{amsthm}
\usepackage{amsfonts}
\usepackage{graphicx}
\usepackage[matrix, arrow, curve]{xy}
\usepackage{comment}
\usepackage{xcolor}
\usepackage[scaled]{helvet} 
\usepackage{courier} 
\usepackage[T1]{fontenc}                                                    
\usepackage{hyperref}

\normalfont

\numberwithin{equation}{section}

\newcommand{\cA}{\mathcal{A}}

\newcommand{\cP}{\mathcal{P}}

\newcommand{\F}{\mathbb{F}}

\newcommand{\bR}{\mathbb{R}}

\newcommand{\R}{\mathbb{R}}
\newcommand{\C}{\mathbb{C}}

\newcommand{\ra}{\rightarrow}

\newcommand{\qor}{\quad \textrm{or} \quad}
\newcommand{\qand}{\quad \textrm{and} \quad}

\newcommand\subsetsim{\mathrel{%
\ooalign{\raise0.2ex\hbox{$\subset$}\cr\hidewidth\raise-0.8ex\hbox{\scalebox{0.9}{$\sim$}}\hidewidth\cr}}}
\newcommand{\eps}{\varepsilon}

\DeclareMathOperator{\supp}{supp}

\newcommand{\Q}{\mathbb Q}
\newcommand{\Z}{\mathbb Z}

\theoremstyle{theorem}
\newtheorem{theorem}{Theorem}[section]
\newtheorem{corollary}[theorem]{Corollary}
\newtheorem{proposition}[theorem]{Proposition}
\newtheorem{lemma}[theorem]{Lemma}

\theoremstyle{definition}
\newtheorem{definition}[theorem]{Definition}
\newtheorem{convention}[theorem]{Convention}
\newtheorem{remark}[theorem]{Remark}
\newtheorem*{example}{Example}

\patchcmd{\subsection}{-.5em}{.5em}{}{}
\patchcmd{\subsubsection}{-.5em}{.5em}{}{}

\begin{document}

\title[Aperiodic order and spherical diffraction]{Aperiodic order and spherical diffraction, I:\\ Auto-correlation of model sets}

\author{Michael Bj\"orklund}
\address{Department of Mathematics, Chalmers, Gothenburg, Sweden}
\email{micbjo@chalmers.se}
\thanks{}

\author{Tobias Hartnick}
\address{Mathematics Department, Technion, Haifa 32000, Israel}
\curraddr{}
\email{hartnick@tx.technion.ac.il}
\thanks{}

\author{Felix Pogorzelski}
\address{Mathematics Department, Technion, Haifa 32000, Israel}
\curraddr{}
\email{felixp@tx.technion.ac.il}
\thanks{}



\date{}

\dedicatory{}
\begin{abstract} We study uniform and non-uniform model sets in arbitrary locally compact second countable (lcsc) groups, which provide a natural generalization of uniform model sets in locally compact abelian groups as defined by Meyer and used as mathematical models of quasi-crystals. We then define a notion of auto-correlation for subsets of finite local complexitiy in arbitrary lcsc groups, which generalizes Hof's classical definition beyond the class of amenable groups, and provide a formula for the auto-correlation of a regular model set. Along the way we show that the punctured hull of an arbitrary regular model set admits a unique invariant probability measure, even in the case where the punctured hull is non-compact and the group is non-amenable. In fact this measure is also the unique stationary measure with respect to any admissible probability measure.
\end{abstract}
\maketitle

\section{Introduction}
Aperiodic point sets in Euclidean space are a classical object of study in geometry, combinatorics and harmonic analysis. The diffraction theory of such point sets was pioneered by Meyer already in the late 1960s \cite{Meyer-69, Meyer-70, Meyer-72}. However, it came to a wider popularity only in the 1980s, after the discovery of quasi-crystals and the subsequent attempts of physicists, crystallographers and mathematicians to provide mathematical models  \cite{KramerN-84, LevineSteinhardt, GratiasMichel} explaining the icosahedral symmetry in the diffraction picture of certain aluminium-manganese alloys discovered experimentally by Shechtman et al.\@ \cite{Shechtman-84}. Diffraction theory of aperiodic point sets in locally compact abelian groups, sometimes called mathematical quasi-crystals, has remained a popular topic in abelian harmonic analysis ever since. The recent monograph \cite{BaakeG-13} lists several hundred references. Among these, the following were particularly influential on the current article: \cite{Dworkin-93, Hof-95, Schlottmann-99, Moody-97, BaakeL-04}. Further developments in the theory of mathematical quasicrystals (including diffraction theory)
are covered in the works \cite{Hof-98, Solomyak-98, LeeMS-02, Moody-02, Gouere-05, BaakeM-04, BaakeLM-07, LenzS-09, MarklofS-14, MarklofS-15, RichardS-15,  KellerR-15, Strungaru15, MoodySurvey, BaakeHuckStrungaru}.

From the point of view of physics and crystallography, it is natural to restrict the attention to quasi-crystals in $\R^n$, $n \leq 3$. From the mathematical point of view, this restriction is rather unnatural, and mathematical quasi-crystals have since long been studied in Euclidean spaces of arbitrary dimensions and in arbitrary locally compact abelian groups. However, there is no reason to stop at the class of locally compact abelian groups. In fact, in the present article and its sequel \cite{BHP2} we demonstrate that a large part of the diffraction theory of mathematical quasi-crystals can be carried out in the framework of arbitrary locally-compact second countable (lcsc) groups. 
More specifically, our goals in this series of articles are three-fold:
\begin{enumerate}
\item to construct plenty of examples of mathematical quasi-crystals in non-abelian (and even non-amenable) lcsc groups, and to point out some of the new phenomena which appear in this context;
\item to develop a theory of diffraction, which works in our general context, and specializes to the classical theory in the abelian case;
\item to compute in a rather explicit way the (spherical) diffraction of our examples.
\end{enumerate}
Classically, the diffraction of a quasi-crystal in $\R^n$ is defined as the Fourier transform of its so-called auto-correlation measure. Similarly, our diffraction theory for quasi-crystals in non-abelian groups will consist of two parts. In this first part we are going to develop the theory of auto-correlation measures of quasi-crystals in the non-abelian setting. The sequel \cite{BHP2} will then be concerned with the (spherical) Fourier transform of this auto-correlation, which we consider as a natural notion of non-abelian diffraction.

\subsection{Model sets} A special role in the Euclidean theory of quasi-crystals is played by so-called model sets, as introduced by Meyer \cite{Meyer-69, Meyer-70, Meyer-72}. These are aperiodic sets constructed by a cut-and-project scheme from lattices in products of locally compact abelian groups. We extend this definition to arbitrary lcsc groups in Definition \ref{def_regular} below after some preliminary definitions. For the purposes of this introduction it suffices to know that a model set is constructed starting from a lattice $\Gamma$ in a product $G\times H$ of lcsc groups by choosing a compact subset $W_0 \subset H$ (called the \emph{window}) and projecting $(G\times W_0) \cap \Gamma$ to $G$. Under certain technical assumptions on the quadruple $(G, H, \Gamma, W_0)$ (see Definition \ref{def_regular}) we call the resulting set $P_0$ a \emph{regular model set}. We say that $P_0$ is \emph{uniform} or \emph{non-uniform} depending on whether $\Gamma$ has the corresponding property. Note that all regular model sets in lcsc abelian groups are uniform, hence the existence of non-uniform model sets is a new phenomenon in the non-abelian setting.

An important property of model sets is their finite local complexity: If $P$ is a model set in a lcsc group $G$, then $P^{-1}P$ is locally finite, i.e.\@ closed and discrete. In fact, it is even uniformly discrete. In the Euclidean context one can show that conversely a relatively dense subset $P \subset \R^n$ has a uniformly discrete difference set $P-P$ if and only if it is a relatively dense subset of a model set, cf.\@ 
\cite{Lagarias-96,Moody-97}. There is no analogous classification theorem for non-abelian groups, but model sets can still be defined and studied in complete analogy with the abelian case.

Model sets exist in many unimodular lcsc groups. For example, every semisimple real Lie group admits a model set which is not contained in a lattice. Non-uniform model sets exist in many real and $p$-adic semisimple Lie groups, despite the fact that semisimple p-adic Lie groups never admit non-uniform lattices. Model sets in nilpotent Lie groups are always uniform, and not every nilpotent group contains a (uniform) model set. However, the class of nilpotent Lie groups containing a uniform model set is substantially larger than the class of nilpotent Lie groups containing a lattice \cite{BH}. For more on these and other examples see Subsection \ref{SubsecExamples}.

\subsection{A general framework for auto-correlation} We now present a general framework for auto-correlation of subset of finite local complexity in lcsc groups, which we are going to apply to model sets. The standard definition of the auto-correlation of a quasi-crystal in $\R^n$ is due to Hof \cite{Hof-95}. Let us first consider the case where $P_0 \subset \R^3$ is a subset of finite local complexity with the property that the finite sums
\[
\sigma_t(f) := \frac{1}{{\rm Vol}(B_t(0))}\sum_{x \in P_0 \cap B_t(0)} \sum_{y \in P_0 \cap B_t(0)} f(y-x)
\] 
converge as $t \to \infty$ for every $f \in C_c(\R^3)$, where $B_t(0)$ denotes the Euclidean ball of radius $t$ around the origin. For example, these assumptions are satisfied if $P_0$ is a model set in $\R^3$. In this situation, the \emph{Hof diffraction} of $P_0$ can be defined as the Radon measure  $\eta_{P_0}$ on $\R^3$ given by
\[
\eta_{P_0}(f) =  \lim_{t \to \infty} \sigma_t(f) \quad (f \in C_c(\R^3)).
\]
If the finite sums $\sigma_t(f)$ do not converge, then one can consider the different accumulation points, but the significance of such arbitrary accumulation points is rather unclear. 

Hof's definition of auto-correlation can be generalized to certain classes of subsets (including uniform regular model sets) of amenable lcsc groups by replacing the sequence $B_t(0)$ by a suitable F\o lner sequence $(F_t)$. Apart from the question of dependence on the F\o lner sequence (which can be resolved), this does yield a reasonable theory of auto-correlation, but this approach has no chance to be generalized beyond amenable groups. We thus suggest an alternative definition of auto-correlation, which works in greater generality. 

To define this notion, we recall that the collection  $\mathcal C(G)$ of closed subsets of a lcsc group $G$ carries a natural compact metrizable topology called the \emph{Chabauty--Fell topology} (see Subsection \ref{SubsecHull}). We consider $\mathcal C(G)$ as a $G$-space with respect to the left-translation action and, given a subset $P_0 \subset G$ of finite local complexity, we denote by  $X = X_{P_0}$ the orbit closure of $P_0$ in $\mathcal C(G)$, which we refer to as the \emph{hull} of $P_0$. We also consider the \emph{punctured hull} $X^\times = X^\times_{P_0}$ given by $X^\times := X \setminus \{\emptyset\}$ and define the periodization map of $P_0$ by the formula
\[
\mathcal P: C_c(G) \to C_0(X^\times_{P_0}), \quad \mathcal Pf(P) := \sum_{x \in P} f(x),
\]
where $C_0^\times(X_{P_0})$ denotes the space of continuous functions on 
$X^\times_{P_0}$ vanishing at infinity. If $\nu$ is a $G$-invariant probability measure on $X_{P_0}$, then there is a unique positive-definite Radon measure $\eta_{\nu}$ on $G$ such that
\[\eta_{\nu}(f^* \ast f) = \|\mathcal P f\|_{L^2(X^\times, \nu)}^2 \quad (f \in C_c(G)),\]
and we refer to this measure as the \emph{auto-correlation} of $\nu$. In many examples of interest, including all regular model sets, there is in fact a unique $G$-invariant measure on $X^\times$; in this case we call $\eta_{P_0} := \eta_{\nu}$ simply the \emph{auto-correlation} of $P_0$. 

For many FLC subsets of amenable lcsc groups, including the case of uniform regular model sets in such groups, the above definition is equivalent to Hof's definition. For example, assume that $G$ is amenable and that $P_0 \subset G$ if an FLC subset whose punctured hull $X^\times = X^\times_{P_0}$ is compact and uniquely ergodic. Then for all $f \in C_c(G)$ we have
\begin{equation}\label{HofLimit}
\eta_{P_0}(f) =\lim_{t \ra \infty} \frac{1}{m_G(F_t)} \sum_{x \in P_0 \cap F_t} \sum_{y \in P_0} f(x^{-1}y) = \lim_{t \ra \infty} \frac{1}{m_G(F_t)} \sum_{x \in P_0 \cap F_t} \sum_{y \in P_0 \cap F_t} f(x^{-1}y)
\end{equation}
along suitable F\o lner sequences $(F_t)$. For example, this holds if $(F_t)$ is a nested van Hove sequence as defined e.g. in \cite{BaakeG-13}, or slightly more generally a weakly admissible F\o lner sequence as defined in Definition \ref{DefWeaklyAdmissible} below. The proof of this fact is a straight-forward generalization of the classical argument in the abelian case.

Remarkably, a version of the first equality in \eqref{HofLimit} still holds in many non-amenable lcsc groups, despite the fact that such groups do not admit any F\o lner sequences whatsoever. We provide some explicit examples of this phenomenon in Theorem \ref{ThmLieGroupsIntro} below. For a more systematic treatment of approximation theorems for non-amenable groups in the context of the so-called \emph{spherical auto-correlation} we refer the reader to the sequel \cite{BHP2} 





\subsection{The hull of a regular model set} From now on let $P_0 \subset G$ be a regular model set associated with a quadruple $(G, H, \Gamma, W_0)$ as above. If $G$ is non-amenable and/or the punctured hull $X^\times = X^\times_{P_0}$ is not compact, then it is not a priori clear whether $X^\times$ admits a $G$-invariant probability measure, and even if such a measure exists, it need not be unique a priori. In the context of non-amenable groups it is actually more natural to consider \emph{stationary} measures on the punctured hull (see the discussion preceding Theorem \ref{ThmUE}) and discuss existence and uniqueness of such measures.

In order to study establish both the existence and the uniqueness of stationary or invariant measures on $X^\times$ we introduce a certain parametrization map between the punctured hull and the parameter space $Y:=(G \times H)/\Gamma$, which in the abelian case reduces to Schlottmann's generalized torus parametrization \cite{Robinson, Schlottmann-99}.
The existence of such a parametrization map yields immediately the following results; here $m_Y$ denotes the unique $(G\times H)$-invariant probability measure on $Y$.
 \begin{theorem}[Punctured hulls of regular model sets]
\label{main1} Let $P_0$ be a regular model set and let $X^\times := X^\times_{P_0}$. Then there exists a unique $G$-invariant probability measure $\nu$ on $X^\times$, which is also the unique stationary probability measure with respect to any admissible probability measure $\mu$ on $G$. Moreover, $(X^\times, \nu)$ is measurably $G$-isomorphic to $(Y, m_Y)$. In particular, as unitary $G$-representations,
\[ L^2(X^\times, \nu) \cong L^2(Y, m_Y).\]
Moreover, if $P_0$ is uniform, then $X^\times = X$ is a compact minimal $G$-space.
\end{theorem}
Note that if $P_0$ is non-uniform, then both $X^\times$ and $Y$ are non-compact. In fact, overcoming this non-compactness is one of the major technical issues in the proof of the theorem.

\subsection{Auto-correlation of regular model sets in semisimple Lie groups} It follows from Theorem \ref{main1} that for every regular model set $P_0$ in a lcsc group $G$ the auto-correlation $\eta_{P_0}$ is well-defined. If $G$ is amenable, then it can equivalently be defined by a limit over F\o lner sets as in \eqref{HofLimit}. A similar approximation formula also holds if $G$ is a semisimple Lie group (and in particular non-amenable). Let us start with a concrete example:
\begin{example}[Approximation formula for regular model sets in ${\rm PSL}_2(\R)$] Let $\mathbb H^2$ be the upper half space model of the hyperbolic plane, and denote by $\pi: T^1\mathbb H^2 \to \mathbb H^2$ its unit tangent bundle. Given $t>0$ we write $B_t \subset \mathbb H^2$ for the hyperbolic ball of radius $t$ around $i$. The group $G := {\rm PSL_2}(\R)$ acts by diffeomorphisms on $\mathbb H^2$, and the induced action on $T^1\mathbb H^2$ is simply transitive. Let $F_t$ be the subset of $G$ which under the diffeomorphism $G \cong T^1\mathbb H^2$ corresponds to $\pi^{-1}(B_t)$. Then for every regular model set $P_0 \subset G$ the auto-correlation measure $\eta_{P_0}$ is given by 
\begin{equation}\label{SemisimpleFormula}
\eta_{P_0}(f) = \lim_{t \ra \infty} \frac{1}{m_{G}(F_t)} \sum_{x \in P_o \cap F_t} 
\sum_{y \in P_o}
f(x^{-1}y) \quad (f \in C_c(G)).\
\end{equation}
\end{example}
The formula carries over to arbitrary semisimple Lie groups as follows:
\begin{theorem}[Approximation formula for regular model sets in semisimple Lie groups]\label{ThmLieGroupsIntro}
Let $G$ be a connected semisimple real Lie group with no non-trivial
compact factors, and $K<G$ a maximal compact subgroup. Let 
$d$ denote the $G$-invariant Riemannian metric on $G/K$ associated to the Cartan-Killing 
form of the symmetric space $G/K$, and for $t>0$ set 
\[
F_t = \big\{ g \in G \, : \, d(K,gK) \leq t \big\}.
\]
Then for every regular model set $P_0 \subset G$ the auto-correlation measure $\eta_{P_0}$ is given by \eqref{SemisimpleFormula}. 
\end{theorem}

\subsection{A general formula for the auto-correlation of regular model sets} 

We conclude this first part with an explicit formula for the auto-correlation of a regular model set $P_0 \subset G$ associated with a quadruple $(G, H, \Gamma, W_0)$ as above in terms of the associated parameter space $Y := (G \times H)/\Gamma$ the associated parameter space. This formula will be used in the sequel \cite{BHP2} to relate the spherical diffraction of a regular model set to the automorphic spectrum of the underlying lattice. Given a Riemann-integrable function $F : G \times H \to \C$ with compact support we denote by
\[
\mathcal P_\Gamma F: Y \to \C, \quad \mathcal P_\Gamma F((g,h)\Gamma) := \sum_{\gamma \in \Gamma} F((g,h)\gamma)
\]
its periodization over $\Gamma$.
\begin{theorem}[Auto-correlation formula for regular model sets]
\label{main2} The auto-correlation $\eta_{P_0}$ of the regular model set $P_0$ as above is uniquely determined by the formula
\begin{equation}
\eta_{P_0}(f^* \ast f) = \|\mathcal P_\Gamma(f\otimes \chi_{W_0})\|^2_{L^2(Y)} \quad (f\in C_c(G)).
\end{equation}
\end{theorem}

\subsection{Organization of the article}

This article is organized as follows. In Section \ref{SecModelHull} we introduce model sets and discuss several classes of examples. We then collect some basic structural properties concerning the hull of a regular model set. This knowledge is then applied in Section \ref{SecHullParam} to construct the parametrization map of such a hull and to deduce unique ergodicity and (in the uniform case) minimality. In the construction of the parametrization map we use an alternative description of the Chabauty--Fell topology on the hull of a model set, which is discussed in detail in Appendix \ref{AppendixLocalTopology}. Section \ref{SecAutocorrelation} is devoted to the construction of the auto-correlation measure and the proof of Theorem \ref{main2}. Finally, in Section \ref{SecApprox} we establish several approximation theorems for the auto-correlation, including Theorem \ref{ThmLieGroupsIntro}.

\subsection{Notational conventions} The following notational conventions will be applied throughout this article:

All function spaces are complex-valued and all inner products are anti-linear in the first variable. Given a locally compact space $X$ we denote by $C_c(X)$, $C_0(X)$ and $C_b(X)$ the function space of compactly supported continuous functions, continuous functions vanishing at infinity and continuous bounded functions respectively.

Given a group $G$ and a function $f: G \to \C$ we denote by $\bar f$, $\check f$ and $f^*$ respectively the functions on $G$ given by \[\bar f(g) := \overline{f(g)}, \quad \check f(g) := f(g^{-1}) \quad \text{and} \quad f^*(g) := \overline{f(g^{-1})}.\] Given an action of $G$ on a set $Z$ we define a $G$-action on complex-valued functions on $Z$ by $g.f(z) := f(g^{-1}.z)$. Moreover we denote by $Z^G \subset Z$ the subset of $G$-invariants. 

If $(X, \nu)$ is a measure space and $f,g \in L^2(X, \nu)$, then we denote by \[\langle f, g \rangle_{X} := \langle f, g \rangle_{(X, \nu)} := \int_X \bar f \cdot g \, d\nu\]  the $L^2$-inner product. Given a subset $A \subset X$ we denote by $\chi_A$ its characteristic function. 

If $G$ is a locally compact, second countable group, then we denote by $m_G$ some fixed choice of left-Haar measure on $G$ (normalized to total mass $1$ in the compact case). We
then denote by $\mathcal{C}(G)$, $\mathcal O(G)$ and $\mathcal K(G)$ the sets of closed, open and compact subsets of $G$ respectively. We also denote by $e$ the identity element of $G$ and by $\mathfrak U(G) = \mathfrak U_e(G)$ the identity neighbourhood filter of $G$. 

\subsection*{Acknowledgement.} We thank  Amos Nevo and Barak Weiss for numerous discussions and their interest in our work. We also profited from discussions with Michael Baake, Uri Bader, Daniel Lenz, Christoph Richard, Eitan Sayag, Boris Solomyak and Andreas Str\"ombergsson. We also thank the mathematics departments of Chalmers University and Technion, Haifa for their hospitality during our respective visits. Finally we thank the organizers and the participants of the Oberwolfach Arbeitsgemeinschaft on ``Mathematical quasi-crystals'' for the opportunity to present and discuss our work.

\section{Model sets and their hulls} \label{SecModelHull}

\subsection{Delone sets in groups}
Let $(X, d)$ be a metric space.  Given $R>0$ and $x \in X$ we denote by $B_R(x)$ the open $d$-ball of radius $R$ around $x$, and given $A \subset G$ we denote by $N_R(A) := \bigcup_{x \in A} B_r(A)$ the $R$-neighbourhood of $A$. \begin{definition}\label{Delone}
Let $r, R >0$. A non-empty subset $P \subset X$ is called
\begin{enumerate}
\item $r$-\emph{uniformly discrete} if $d(x, y) \geq r$ for all $x,y \in P$;
\item $R$-\emph{relatively dense} if $N_R(P) = G$;
\item a $(r, R)$-\emph{Delone set} if it is both $r$-uniformly discrete and $R$-relatively dense.
\end{enumerate}
It is called \emph{uniformly discrete}, \emph{relatively dense} or a \emph{Delone set} if it has the respective property for some $r, R>0$.
\end{definition}
Given a locally-compact second countable (lcsc) group $G$, let us call a metric $d$ on $G$ \emph{admissible} if it is proper (i.e. closed balls are compact), left-invariant and induces the given topology on $G$. We call a subset $P \subset G$ \emph{uniformly discrete/relatively dense/a Delone set} if it has the corresponding property for some admissible metric on $G$, in which case it actually has the corresponding property for any admissible metric on $G$. These properties then admit the following purely topological characterizations (see e.g. \cite[Prop. 2.2 and Lemma 2.3]{BH}): $P \subset G$ is uniformly discrete if and only if  the identity $e \in G$ is not an accumulation point of $P^{-1}P$. Equivalently, there exists an open subset $V \subset G$ such that $|P \cap gV| \leq 1$ for all $g \in G$. (We then say that $P$ is \emph{$V$-discrete}.) On the other hand, $P$ is relatively dense if and only if it is \emph{left-syndetic}, meaning that 
there exists a compact subset $K \subset G$ such that $PK  = G$. 

In our further study of Delone sets in lcsc groups we will need the following related notions of discreteness:
\begin{definition} 
Let $G$ be a lcsc group and $P \subset G$ be a subset.
\begin{enumerate}
\item $P$ is called \emph{locally finite} if it is closed and discrete. 
\item $P$ is called \emph{(left-)uniformly locally finite} if for some (hence any) compact subset $K \subset G$ there exists $C>0$ such that $|P \cap gK| < C$ for all $g \in G$. (We then say that $P$ is \emph{$(K,C)$-locally finite.}) 
\item $P$ has \emph{(left-) finite local complexity} (FLC) if $P^{-1}P$ is locally finite.
\end{enumerate}
\end{definition}
We then have the chain of implications
\begin{eqnarray}
&& P \text{ has left-FLC} \,\Rightarrow\, P^{-1}P \text{ discrete} \,\Rightarrow\, P \text{ uniformly discrete} \nonumber\\
&&  \,\Rightarrow\, P \text{ left-uniformly locally finite} \,\Rightarrow\, P \text{ locally finite} \,\Rightarrow\, P \text{ discrete}.\label{Discreteness}
\end{eqnarray}
Analogous characterizations hold with left replaced by right (where we then have to define uniform discreteness with respect to a right-invariant metric and replace $P^{-1}P$ by $PP^{-1}$).

\subsection{Model sets}\label{SecModel} We recall that a discrete subgroup $\Gamma$ of a lcsc group $G$ is a \emph{lattice} if $G/\Gamma$ admits a $G$-invariant probability measure, and a \emph{uniform lattice} if $G/\Gamma$ is moreover compact. We also recall that every discrete subgroup of a lcsc group is necessarily closed, hence locally finite.

\begin{definition} A \emph{cut-and-project-scheme} is a triple $(G, H, \Gamma)$ where $G$ and $H$ are lcsc groups and $\Gamma < G \times H$ is a lattice which projects injectively to $G$ and densely to $H$. A cut-and-project scheme is called \emph{uniform} if $\Gamma$ is moreover a uniform lattice.
\end{definition}
\begin{remark}[Cut-and-project schemes from irreducible lattices]\label{RemarkMNST} The assumptions on the projections of $\Gamma$ can often be arranged. Assume for example that we start from a lattice $\Gamma_0 < G\times H_0$ which is irreducible in the sense that it projects densely to both factors, and denote by $p_G$ and $p_{H_0}$ the coordinate projections of $G \times H_0$. Then $\Gamma_1 := \ker(p_G|_{\Gamma_0})$ is a normal subgroup of $\Gamma_0$, and $p_{H_0}(\Gamma_1)$ is a normal subgroup of $H_0$. If $\Gamma_1$ is finite, then we obtain a cut-and-project scheme
\[
(G, H_0/p_{H_0}(\Gamma_1), \Gamma_0/\Gamma_1).
\]
In certain situations of interest, finiteness of $\Gamma_1$ holds automatically. For example, if $G$ and $H_0$ are semisimple Lie groups, then $\Gamma_0$ is a higher rank lattice and  $\Gamma_1$ is of infinite index in $\Gamma_0$ since $\Gamma_0/\Gamma_1$ projects densely to $G$. It thus follows from Margulis' normal subgroup theorem that $\Gamma_1$ is finite. 
\end{remark}

Given a cut-and-project scheme $(G, H, \Gamma)$ we denote by $\pi_G$, $\pi_H$ the coordinate projections of $G \times H$ and set $\Gamma_G := \pi_G(\Gamma)$ and $\Gamma_H := \pi_H(\Gamma)$. We then define a map $\tau: \Gamma_G \to H$ as $\tau := \pi_H \circ (\pi_G|_\Gamma)^{-1}$. Note that the image of $\tau$ is precisely $\Gamma_H$; in the abelian case this map is sometimes called the ``$*$-map''.

\begin{definition} Let  $(G, H, \Gamma)$ be a cut-and-project scheme with associated ``$*$-map'' $\tau: \Gamma_G \to H$. Given a compact subset $W_0 \subset H$, the pre-image
\[
P_0(G, H, \Gamma, W_0) := \tau^{-1}(W_0) \subset G
\]
is called a \emph{weak model set}, and $W_0$ is called its \emph{window}. It is called a \emph{weak uniform model set} if $(G, H, \Gamma)$ is uniform.
\end{definition}
Every weak model set has finite local complexity \cite[Prop. 2.13]{BH}, but it need not be relatively dense, even if it is uniform and $G$ and $H$ are abelian. For example, the set of visible lattice points in $\R^2$ is a weak uniform model set which is not relatively dense, see \cite[Prop. 10.4]{BaakeG-13}. While general weak model sets have received much attention recently, we will focus on the following more classical subclasses.
\begin{definition}\label{def_regular} A weak (uniform) model set is called a \emph{(uniform) model set} if its window has non-empty interior. It is called a \emph{regular model set} if $W_0$ is 
Jordan-measurable with dense interior, aperiodic (i.e. and ${\rm Stab}_H(W_0) = \{e\}$ and \emph{$\Gamma$-regular}, meaning that
\[\partial W_0 \cap \pi_H(\Gamma) = \emptyset.\]
\end{definition}
Some of the theory developed in the sequel could be developed in the generality of arbitrary model sets. However, to obtain our strongest results, such as unique ergodicity and minimality of the hull, we need additional regularity assumptions on the window, hence we make the 
\begin{convention}\label{regular} From now on all model sets (uniform or not) are assumed to be regular.
\end{convention}

\subsection{Examples of model sets}\label{SubsecExamples} Let us provide some examples of model sets and cut-and-project schemes. The classical theory as developed by Meyer \cite{Meyer-70} is concerned with abelian cut-and-project schemes. A particularly interesting class of examples arises from arithmetic lattices in lcsc abelian groups. For example, $\Gamma = \mathbb Z[\sqrt 2]$ embeds as a lattice into $\R \times \R$ via $a+b\sqrt 2 \mapsto (a+b\sqrt 2, a-b\sqrt 2)$ and $\Gamma = \mathbb Z[1/p]$ embeds as a lattice into $\R \times \mathbb Q_p$ leading to different model sets in $\R$. These arithmetic examples admit various generalizations to non-abelian cut-and-project schemes as follows.

\begin{example}[Model sets in nilpotent Lie groups] Denote by $H(R)$ the $3$-dimensional Heisenberg group over a commutative ring $R$, i.e. the subgroup of $M_3(R)$ consisting of upper triangular matrices with $1$s on the diagonal. Then $H(R)$ is $2$-step nilpotent and the (uniform) lattice embeddings $ \mathbb Z[\sqrt 2] \hookrightarrow \R \times \R$ and $\mathbb Z[1/p] \hookrightarrow \R \times \mathbb Q_p$ induce (uniform) lattice embeddings
\[
H(\mathbb Z[\sqrt 2])\hookrightarrow H(\R) \times H(\R) \quad \text{and} \quad H(\mathbb Z[1/p]) \hookrightarrow H(\R) \times H(\mathbb Q_p),
\]
leading to different kinds of uniform model sets in the real Heisenberg group $H(\R)$. 
\end{example}
\begin{remark} One major difference between the abelian and the nilpotent case is that while every abelian Lie group admits a (uniform) lattice, there exist nilpotent Lie groups which do not admit any (uniform) lattices, but which do admit (uniform) model sets. For explicit examples in dimension 7, see \cite[Sec. 2.4]{BH}.
\end{remark}
While every model set in an abelian or nilpotent Lie group is uniform, this is not the case for general amenable group. In fact, the following example which modifies a construction of   Bader, Caprace, Gelander and Mozes \cite[Example 3.5]{BenoistQuint} shows that there even exist non-uniform model sets in compact-by-abelian groups.

\begin{example}[A non-uniform model set in a compact-by-abelian group]
Let $S$ be a set of primes which is ``thin'' in the sense that $\sum_{p \in S} \frac 1 p < \infty$ and set \[V := \bigoplus_{ p \in S} \F_p \quad \text{and} \quad K:= \prod_{p \in S} \F_p^\times,\]
so that $V$ is abelian and $K$ is compact and acts by coordinate-wise multiplication on $V$. Set $G:= V \rtimes K$ and, given $\gamma \in G$, write $\gamma = (\gamma_p)_{p \in S}$, where $\gamma_p \in  \F_p \rtimes \F_p^\times$. We claim that $G$ contains regular non-uniform model sets. For this we construct a cut-and-project scheme $(G, H, \Gamma)$ with $H := G$ and $\Gamma$ given as follows. Let $\Gamma_0 < G$ be the dense subgroup given by
\[
\Gamma_0 := \bigoplus_{p \in S} \F_p \rtimes \F_p^\times < G,
\]
and define a dense embedding $\tau: \Gamma_0 \to H$ by
\[
\tau((b_p, a_p)_{p \in S})  = ((1,1) (b_p, a_p)(1,1)^{-1})_{p \in S} = (b_p+1-a_p, a_p)_{p \in S}.
\]
Then we set 
\[\Gamma := \{(\gamma, \tau(\gamma))\mid \gamma \in \Gamma_0\} < G\times H.\]
By construction, $\Gamma$ projects injectively to $G$ and densely to $H$. It thus remains to show that it is a non-uniform lattice in $G \times H$. Discreteness of $\Gamma$ is an easy exercise. Unravelling definitions we see that two points $(u,v), (u',v') \in V \oplus V$ are in the same $\Gamma$-orbit if and only if for every $p \in S$ there exist elements $b_p \in \F_p$ and $a_p \in \F_p^\times$ such that
\begin{equation*}\label{OrbitEquation}
b_p + a_pu_p = u'_p \quad \text{and}\quad b_p+1-a_p + a_pv_p = v_p'.
\end{equation*}
From these formulas it is not hard to deduce that the $\Gamma$-orbits in $V \oplus V$ are exactly the sets
\[
S_I := \{(u,v)\mid v_p-u_p = 1 \Leftrightarrow p \in I\},
\]
where $I \subset S$ is a finite subset. In particular, $\Gamma$ is not cocompact, since there are infinitely many such orbits. To show that $\Gamma$ has finite covolume it suffices to show that
\[
\sum_{[(u, v)] \in \Gamma \backslash V \oplus V} \frac{1}{|\Gamma_{(u,v)}|} < \infty,
\]
where $\Gamma_{(u,v)}$ denotes the stabilizer of $(u,v) \in V \oplus V$ in $\Gamma$ and the sum is over all $\Gamma$-orbits in $V \oplus V$ (see \cite[Sec. 1.5]{BassLubotzky}). 
Given $(u,v) \in S_I$ we have
\[
\Gamma_{(u,v)} = \bigoplus_{p \in S}\{(b_p, a_p) \in \F_p\times \F_p^\times \mid b_p + a_pu_p = u_p \quad \text{and}\quad b_p+1-a_p + a_pv_p = v_p\},
\]
from which one deduces that $|\Gamma_{(u,v)}| = \prod_{p \in I}(p-1)$. We obtain
\[
\sum_{[(u, v)] \in \Gamma \backslash V \oplus V} \frac{1}{|\Gamma_{(u,v)}|}  = \underset{\text{finite}}{\sum_{I \subset S}} \frac{1}{\prod_{p \in I}(p-1)} < \infty,
\]
which shows that $\Gamma$ has finite covolume in $G \times H$.
\end{example}
We now turn to model sets in non-amenable group. 
\begin{example}[Model sets in real semisimple Lie groups]
Similarly to the abelian case one can construct arithmetic examples of cut-and-project in semisimple Lie groups. If $G$ is an arbitrary semisimple Lie group without compact factors, then by \cite[Cor. 18.7.4]{WitteMorris} there always exists an irreducible uniform lattice in $G\times G$. By Remark \ref{RemarkMNST} this implies that $G$ contains uniform model sets. Similarly, non-uniform model sets can be constructed starting from irreducible non-uniform lattices such as
\[
{\rm SL}_n( \mathbb Z[\sqrt 2]) < {\rm SL}_n(\R) \times {\rm SL}_n(\R) \quad \text{and} \quad {\rm SL}_n(\mathbb Z[1/p])<{\rm SL}_n(\R) \times {\rm SL}_n(\mathbb Q_p).
\]
\end{example}

\begin{example}[Non-unimodular model sets in semisimple p-adic groups]
Starting from the irreducible lattice $ {\rm SL}_n(\mathbb Z[1/p])< {\rm SL}_n(\mathbb Q_p) \times {\rm SL}_n(\R) $ we can construct non-uniform model sets in ${\rm SL}_n(\Q_p)$. This is remarkable, since ${\rm SL}_n(\Q_p)$ does not admit any non-uniform lattices.
\end{example}

\begin{example}[A geometric example] There exist examples of model sets in geometrically defined non-amenable totally-disconnected lcsc groups, such as automorphism groups of regular trees. These examples are very different in nature from the arithmetically flavored examples above. For example, while all arithmetic lattices are residually finite, Burger and Mozes \cite{BurgerM-001, BurgerM-002} have constructed lattices in products of automorphism groups of regular trees which are simple (and thus as far from residually finite as possible), and these lead to interesting examples of model sets.
\end{example}
\begin{example}[Hartman sets] Yet another very different (but well-understood) class of examples arises from discrete groups $\Gamma$ which embed densely into a compact group $K$. We can then view $\Gamma$ as a lattice in $G \times H$, where $G := \Gamma$ and $H:= K$ and study the corresponding model sets, which are called \emph{Hartman sets}. We refer the interested reader to the survey \cite{Winkler}.
\end{example}

\subsection{The hull of a subset of finite local complexity}\label{SubsecHull} If $G$ is a lcsc group, then there is a natural compact metrizable topology $\tau_{CF}$ on the collection $\mathcal C(G)$ of closed subsets of $G$ which is called the \emph{Chabauty--Fell topology} and defined as follows. Recall that we denote by  $\mathcal O(G)$ and $\mathcal K(G)$ the collections of  open, respectively compact subsets of $G$. Given $V \in \mathcal O(G)$ and $K \in \mathcal K(G)$ we define subsets $U_V, U^K \subset \mathcal C(G)$ by
\[
U_V = \{C \in \mathcal C(G)\mid C \cap V \neq \emptyset\} \quad \text{and} \quad U^K = \{C \in \mathcal C(G)\mid C \cap K = \emptyset\}.
\] 
Then the Chabauty--Fell topology on $\mathcal C(G)$ is generated by the collection of basic open sets
\[\{U_V\mid{V \in \mathcal O(G)}\} \cup \{U^K\mid{K \in \mathcal K(G)}\}.\]
The group $G\times G$ acts on the compact space $\mathcal C(G)$ by
\[
(g,h).\Lambda = g\Lambda h^{-1},
\]
and this action is jointly continuous, since it maps basic open sets to basic open set:
\[
(g,h).U_{V} = U_{gVh^{-1}}, \quad (g,h).U^K = U^{gKh^{-1}} \quad (g,h\in G).
\]
Restricting the action of $G \times G \curvearrowright \mathcal C(G)$ to the factors and the diagonal we obtain three topological dynamical systems over $G$, where $G$ acts from the left, the right or by conjugation. The former two dynamical systems are isomorphic via the isomorphisms $P \mapsto P^{-1}$, but the conjugation system has very different properties. Here we will focus on the action of $G$ on the left as given by $(g, P) \mapsto gP$ for $g \in G$, $P \in \mathcal C(G)$.

\begin{definition} Given a closed subset $P_0 \subset G$, the \emph{(right-)hull} $X_{P_0}$ of $P_0$ is defined as the closure of the orbit $G.P_0$ in $\mathcal C(G)$, and the \emph{punctured (right-)hull} $X_{P_0}^\times$ is defined as $X_{P_0}^\times := X_{P_0} \setminus\{\emptyset\}$.
\end{definition}
Note that by definition the hull of a closed subset is always a compact metrizable $G$-space, since it is a closed subset of $\mathcal C(G)$, and thus $X_{P_0}^\times$ is always locally compact and metrizable. By \cite[Prop. 4.4]{BH} we have $\emptyset \in X_{P_0}$ if and only if $P_0$ is not relatively dense. In particular the punctured hull $X_{P_0}^\times$ of a model set $P_0$ is compact if and only if $P_0$ is uniform. While the hull of a uniformly finite subset of $G$ may contain non-discrete subsets of $G$, the hull of an FLC subset always consists of FLC sets:
\begin{proposition}[Orbit closures of FLC sets, {\cite[Lemma 4.6]{BH}}] \label{PropHull} Let $G$ be a lcsc group and let $P_0 \subset G$ be an FLC subset. Then for every $P \in X_{P_0}$ we have 
\[
P^{-1}P \subset P_0^{-1}P_0.
\]
In particular, every $P \in X_{P_0}$ has finite local complexity, and if $\Gamma_G$ denotes the subgroup of $G$ generated by $P_0$ and  $P \in X_{P_0}$, then $p^{-1}P \subset \Gamma_G$ for every $p \in P$.\qed
\end{proposition}

\subsection{The hull of a model set and the canonical transversal} Since model sets have finite local complexity, the discussion of the previous subsection applies in particular to model sets. Given a model set $P_0 = P_0(G, H, \Gamma, W_0)$, the subgroup of $G$ generated by $P_0$ is given by $\Gamma_G = \pi_G(\Gamma)$. It thus follows from Proposition \ref{PropHull} that the subset
\begin{equation}\label{CanTransv}
\mathcal T := \{P \in X^\times_{P_0} \mid P \subset \Gamma_G\} \quad \subset \quad X^\times_{P_0}.
\end{equation}
intersects every $G$-orbit in $X_{P_0}$. We refer to $\mathcal T$ as the \emph{canonical transversal} in $X^\times_{P_0}$. We can summarize our discussion so far as follows:
\begin{corollary}\label{HullModel} Let $P_0 = P_0(G, H, \Gamma, W_0)$ be a model set in a lcsc group $G$. 
\begin{enumerate}[(i)]
\item $X^\times := X^\times_{P_0}$ is a locally compact metrizable $G$-space consisting of FLC sets. 
\item Every $G$-orbit in $X^\times$ intersects the canonical transversal $\mathcal T$ given by \eqref{CanTransv}. 
\item $X^\times$ is compact if and only if $\Gamma < G \times H$ is a uniform lattice.\qed
\end{enumerate}
\end{corollary}

\section{Minimality and unique ergodicity of the punctured hull}\label{SecHullParam}

\subsection{The parametrization map of a model set} 
The goal of this section is to show that the punctured hull of a model set is uniquely ergodic (and moreover minimal if the model set is uniform). This will be achieved by means of a parametrization map, which generalizes the well-known torus parametrization of the abelian theory. This parametrization map is of independent interest, and we will use it to derive a formula for the auto-correlation of the corresponding model set in Section \ref{SecAutocorrelation}.

Throughout this section we fix the following setting. $P_0 = P_0(G, H, \Gamma, W_0)$ is a model set in a lcsc group $G$ (always assumed to be regular by Convention \ref{regular}). We denote by $X^\times = X^\times_{P_0}$ its punctured hull, and by $\mathcal T \subset X^\times$ the canonical transversal given by \eqref{CanTransv}. Finally, we abbreviate by $Y$ the homogeneous space
\[
Y := (G \times H)/\Gamma,
\]
which we refer to as the \emph{parameter space} of $P_0$ and abbreviate by $y_0 := (e,e)\Gamma$ the canonical basepoint of $Y$. The following theorem summarizes the main results of this subsection. 
\begin{theorem}[Properties of the parametrization map]\label{ThmParametrizationMap} There exists a unique $G$-equivariant Borel map $\beta: X^\times \to Y$ which maps $P_0$ to $y_0$ and has a closed graph. This map has the following additional properties:
\begin{enumerate}[(i)]
\item If $Y^{\rm ns} := \{(g,h)\Gamma \in Y\mid h^{-1}W_0 \text{ is }\Gamma\text{-regular}\}$ and $X^{\rm ns} := \beta^{-1}(Y^{\rm ns})$, then \[\beta|_{X^{\rm ns} }: X^{\rm ns} \to Y^{\rm ns}\] is bijective.
\item $P \in \mathcal T$ if and only if there exists $h_P \in H$ such that $\beta(P) = (e, h_P)\Gamma$.
\item If $P \in \mathcal T \cap X^{\rm ns}$ and $h_P$ is as in (ii), then $P = \tau^{-1}(h_P^{-1}W_0)$.
\item It $\Gamma$ is cocompact, then $\beta$ is continuous.
\end{enumerate}
\end{theorem}
\begin{remark}
\begin{enumerate}
\item In view of (i), elements of $X^{\rm ns}$ are called \emph{non-singular points} and elements of $Y^{\rm ns}$ are called \emph{non-singular parameters}.
\item The special case $G = \mathbb R^k$, $H = \mathbb R^n$ is classical. In this case, $Y$ is a torus and $\beta: X^\times \to Y$ is known as the \emph{torus parametrization} of the hull. For general locally compact abelian groups $G$, $H$ the construction of a parametrization map $\beta$ is due to Schlottmann \cite{Schlottmann-99} (see also \cite{Robinson} for an earlier special case). In his proof he first establishes minimality of $X^\times$ using compactness (in the form of Gottschalk's criterion), and then uses minimality to establish existence of the parametrization map. 
\item Unlike Schlottmann's proof, the present argument does not require compactness of $X^\times$, nor any a priori knowledge of minimality. Consequently, the argument also applies to non-compact punctured hulls, and minimality comes for free in the case of compact punctured hulls.
\item Our proof of Theorem \ref{ThmParametrizationMap} does not use the full assumptions on $\Gamma$ and $W_0$. We do not need that $\Gamma$ is a lattice (as long as it is discrete and satisfies the other assumptions), nor do we use that $W_0$ is Jordan-measurable. However, both assumptions will be used in the sequel. We need that $W_0$ is Jordan-measurable to obtain that $Y^{\rm ns}$ has full Haar measure in $Y$, and that $\Gamma$ is a lattice to obtain an invariant probability measure on $X$. We therefore do not pursue this additional generality here.
\end{enumerate}
\end{remark}
\subsection{Minimality and unique ergodicity of the hull} 
Let us postpone the proof of Theorem \ref{ThmParametrizationMap} to the end of this section and first deduce the desired consequences concerning minimality and unique ergodicity of the hull. We keep the notation of the previous subsection and recall that the model set $P_0$ is relatively dense in $G$ provided it is a uniform model set.
\begin{proposition}[Minimality properties]\label{Minimality}
The hull of a uniform regular model set is minimal. More precisely:
\begin{enumerate}[(i)]
\item $Y$ is a minimal $G$-space.
\item If $P_0$ is uniform, then $X^\times = X$ is a minimal compact $G$-space.
\item If $P_0$ is non-uniform, then $X^\times$ is not compact and no $G$-orbit in $X^\times$ is pre-compact.
\end{enumerate}
\end{proposition}
\begin{proof} Observe first that since $\Gamma_H$ is dense in $H$, the space $G\backslash (G\times H)$ is minimal as a $\Gamma$-space, and thus $Y = (G\times H)/\Gamma$ is minimal as a $G$-space by the duality principle. Now assume that $X_0 \subset X^\times$ is a non-empty compact $G$-invariant subset. Since $\beta$ has a closed graph, it maps compact sets to closed sets, and since it is $G$-equivariant, the image of $X_0$ is a closed $G$-invariant subset of $Y$. Minimality of $Y$ then yields $\beta(X_0) = Y$.  In particular $\beta(P_0) \in \beta(X_0)$, and since $P_0 \in X^{\rm ns}$ we deduce $P_0 \in X_0$ and thus $X_0 =X^\times$. This proves that every compact $G$-invariant subset of $X^\times$ is either empty or all of $X^\times$. Since the punctured hull of a closed subset is compact if and only if the subset if relatively dense, this implies (ii) and (iii).
\end{proof}
We now turn to the question of unique ergodicity of the hull. It it is an immediate consequence of the duality principle that $(Y, m_Y)$ is (uniquely) $G$-ergodic, and we want to lift this property to $X^\times$. The natural context to discuss this problem is that of stationary measures. Recall that a probability measure $\mu$ on $G$ is called \emph{admissible} if its support generates $G$ and if it is absolutely continuous with respect to the Haar measure $m_G$ on $G$. Then a probability measure $\nu$ on a measurable $G$-space is called $\mu$-\emph{stationary} provided  $\mu \ast \nu = \nu$. If $G$ is a non-amenable group and $Z$ is a compact $G$-space, then there might not exist $G$-invariant measures, but by a straight-forward application of the Kakutani fixpoint theorem there will always exist a $\mu$-{stationary} probability measure on $Z$ for any admissible probability measure $\mu$ on $G$, and one may then ask whether such a measure is actually unique. 
\begin{theorem}[Existence and uniqueness of stationary and invariant measures on $X^\times$]\label{ThmUE}
There exists a unique $G$-invariant probability measure $\nu$ on $X^\times$. This measure satisfies $\nu(X^{\rm ns}) = 1$ and is also the unique stationary probability measure with respect to any admissible probability measure $\mu$ on $G$. 

Moreover, $\beta: (X^\times, \nu) \to (Y, m_Y)$ is a measurable isomorphism of $G$-spaces and thus induces an isomorphism
\[ \beta^*: L^2(Y, m_Y)  \to L^2(X^\times, \nu)\] of unitary $G$-representations. 
\end{theorem}
In the language of \cite{BH}, Theorem \ref{ThmUE} yields the following conclusion when combined with \cite[Prop. 2.13]{BH}:
\begin{corollary}[Regular model sets as approximate lattices] A regular model set is a strong approximate lattice provided its window is symmetric and contains the identity.\qed
\end{corollary}

The proof of Theorem \ref{ThmUE} is a consequence of Theorem \ref{ThmParametrizationMap} and the following two lemmas.
\begin{lemma}\label{NSgeneric} The subset $Y^{\rm ns} \subset Y$ is conull with respect to Haar measure on $Y$.
\end{lemma}

\begin{lemma}\label{LemmaUE} The Haar measure $m_Y$ is the unique stationary probability measure on $Y$ with respect to any admissible probability measure $\mu$ on $G$.
\end{lemma}

\begin{proof}[Proof of Theorem \ref{ThmUE}] Let $\beta: X^\times \to Y$ be the parametrization map constructed in Theorem \ref{ThmParametrizationMap}. Since $\beta$ has a closed graph, so does the restriction $\beta|_{X^{\rm ns}}: X^{\rm ns} \to Y^{\rm ns}$ and hence also its inverse $(\beta|_{X^{\rm ns}})^{-1}: Y^{\rm ns} \to X^{\rm ns}$. In particular, $(\beta|_{X^{\rm ns}})^{-1}$ is Borel and we may define a $G$-invariant probability measure on $X^\times$ by
\[
\nu := (\beta|_{X^{\rm ns}})^{-1}_*\; m_Y|_{Y^{\rm ns}}.
\]
By definition, we have for every Borel subset $A \subset X^\times$,
\[
\nu(A) = m_Y(\beta(A \cap X^{\rm ns})) = m_Y(\beta(A) \cap Y^{\rm ns}) = m_Y(\beta(A)).
\]
Now let $\mu$ be an admissible probability measure on $G$ and $\nu'$ a $\mu$-stationary probability measure on $X^\times$. Then $\beta_*\nu'$ is a $\mu$-stationary measure on $Y$ and thus $\beta_*\nu' = m_Y$ by Lemma \ref{LemmaUE}. Since $\beta_*\nu'(Y^{\rm ns}) = 1$ by Lemma \ref{NSgeneric} we deduce that $\nu'(X^{\rm ns}) = 1$, i.e. $\nu'$ is a probability measure on $X^{{\rm ns}}$. Now  $\mu$-stationary measures on $X^{\rm ns}$ correspond bijectively via $\beta$ to $\mu$-stationary measures on $Y^{\rm ns}$ via the  $G$-equivariant Borel isomorphism $\beta|_{X^{\rm ns}}$. We conclude that $\nu' = \nu$, and the theorem follows.
\end{proof}

\begin{proof}[Proof of Lemma \ref{NSgeneric}] An element $(g,h)\cdot \Gamma \in Y$ is singular if and only if $h^{-1}W_0$ is not $\Gamma$-regular. This amounts to $h^{-1}\partial W_0 \cap \Gamma_H \neq \emptyset$, i.e. $h \in \partial W_0 \Gamma_H$. Now since $W_0$ is Jordan-measurable we have $m_H(\partial W_0) = 0$ and thus $\partial W_0 \Gamma_H$ is a nullset by countability of $\Gamma_H$.
\end{proof}

\begin{proof}[Proof of Lemma \ref{LemmaUE}] Fix an admissible probability measure $\mu$ on $G$ and let $\nu$ be an arbitrary $\mu$-stationary probability measure on $Y$. We are going to show that $\nu = m_Y$.

For every non-negative function $\rho \in C_c(H)$ normalized to $\int \rho dm_H = 1$ we define a probability measure $\nu_\rho$ on $Y$ by $\nu_\rho := (\mu \otimes \rho m_H) \ast (\mu \otimes \rho m_H) \ast \nu$. Using that the $G$- and $H$-action commute, we see that $\nu_\rho$ is $\mu$-stationary. Since $\mu$ and $\rho m_H$ are respectively absolutely continuous with respect to $m_G$ and $m_H$ we deduce that $ (\mu \otimes \rho m_H) \ast \nu$ is absolutely continuous with respect to $m_Y$. The second convolution then has a smoothing effect, and we deduce that $\nu_\rho$ has a continuous density $\psi_\rho \in C(Y)$ with respect to Haar measure. Since $m_Y$ is $G$-invariant and $\nu_\rho$ is $\mu$-stationary, the density $\psi_\rho$ is $\mu$-stationary as well. By a standard argument, this implies that $\psi_\rho$ is actually $G$-invariant. Indeed, since $\psi_\rho$ is continuous and the support of $\mu$ generates $G$ as a semigroup, it suffices to show that for all $k \in \mathbb N$,
\[
I_k := \int_G \int_Y (\psi_\rho(gy)-\psi_\rho(y))^2dm_Y(y)d\mu^{*k}(g) =0,
\]
Using stationarity of $m_Y$ and $\psi_\rho$ and expanding the square we obtain
\[
I_k = 2\left(\int_Y \psi_\rho^2 dm_Y - \int_Y \psi_\rho(y) \left[\int \psi_\rho(gy)d\mu^{*k}(g) \right]dm_Y(y) \right)=0\]
for all $k \in \mathbb N$. This shows that $\psi_\rho$ is indeed $G$-invariant, hence constant by Proposition \ref{Minimality}.(i). We deduce that $\psi_\rho = 1$ and $\nu_\rho = m_H$ for every $\rho$ as above.

Now let $\rho_n$ be a sequence of normalized positive functions in $C_c( H)$ such that $\rho_nm_H$ converges to $\delta_e$ in the weak-$*$ topology. Then the previous argument yields $\nu_{\rho_n} = m_Y$ for every $n$ and thus
\[
\nu = \lim_{n \to \infty} \nu_{\rho_n} = m_Y. \qedhere
\]
\end{proof}
\subsection{The local topology and the proof of Theorem \ref{ThmParametrizationMap}}
Our proof of Theorem \ref{ThmParametrizationMap} uses a characterization of the Chabauty--Fell topology on the orbit closure of a model set which is established in Appendix \ref{AppendixLocalTopology}, and which we summarize here for the convenience of the reader.

By definition, the \emph{(left-) local topology} is the unique topology on $\mathcal C(G)$ such that for every $P \in \mathcal C(G)$ a neighbourhood basis of $P$ is given by the sets
\[U_{K,V}(P) =  \{Q \in \mathcal C(G) \mid \exists\, t \in V:\; tQ\cap K = P \cap K\},
\]
where $K$ runs over all compact subsets of $G$ and $V$ runs over the identity neighbourhood basis of $G$.

It turns out that the local topology is finer than the Chabauty--Fell topology on $\mathcal C(G)$. However, if $P \in \mathcal C(G)$ is of finite local complexity (e.g. a model set), then by Corollary \ref{TopologiesCoincide} the orbit closures of $P$ in the local topology and the Chabauty--Fell topology coincide. It will thus suffice to establish Theorem \ref{ThmParametrizationMap} for the orbit closure of $P_0 = P_0(G, H, \Gamma, W_0)$ in $\mathcal C(G)$ with respect to the local topology, which is computationally more convenient. For the rest of this section we will exclusively work with the local topology on $\mathcal C(G)$.

Towards the proof of Theorem \ref{ThmParametrizationMap} we first note that the assumptions on $G$ imply that $G$ is $\sigma$-compact. We may thus fix an exhaustion $K_1 \subset K_2 \subset \dots \subset G$ of $G$ by compact subsets. We also fix a sequence of symmetric pre-compact open identity neighbourhoods $V_1 \supset V_2 \supset \dots$ in $G$ such that $\bigcap V_n = \{e\}$. The proof of Theorem \ref{ThmParametrizationMap} will be based on the following two lemmas:
\begin{lemma}\label{LemmaWindowMagic}\label{LemmaRegularityConvergence} 
Let $W_0 \subset H$ be a window, let $h, h' \in H$ and let $(h_n), (h_n')$ be sequences in $H$ converging to $h$ and $h'$ respectively.
\begin{enumerate}[(i)]
\item If $h \neq h'$, then there exists  a non-empty, open subset $U \subset H$ such that for all sufficiently large $n \in \mathbb N$,
\[
U \subset h_nW_0 \setminus h_n'W_0.
\]
\item  If the windows $h_nW_0$ and $hW_0$ are $\Gamma$-regular, then for every $K \in \mathcal K(G)$ there exists $n_0 \in \mathbb N$ such that for all $n \geq n_0$ we have
\[
(K \times h_nW_0) \cap \Gamma = (K \times hW_0) \cap \Gamma.
\]
\end{enumerate}
\end{lemma}

\begin{lemma}\label{LemmahP} For every $P \in \mathcal T$ there exists $h_P \in H$ with the following property. For every sequence $(g_n)$ in $G$ with $g_n P_0 \to P$ there exists a subsequence $(g_{n_i})$ such that
\begin{enumerate}[(i)]
\item $g_{n_i} = s_i \gamma_i$ for some $s_i \in V_i$, $\gamma_i \in \Gamma_G$;
\item $s_i \to e$, $\tau(\gamma_i) \to h_P^{-1}$;
\item For every $i,j \in \mathbb N$ with $ j \geq i$ we have
\[
\gamma_j P_0 \cap K_i = P \cap K_i.
\]
\end{enumerate}
\end{lemma}
Let us first explain how these lemmas imply the theorem.
\begin{proof}[Proof of Theorem \ref{ThmParametrizationMap}] Consider the orbit closure $Z := \overline{G.(P_0, y_0)} \subset  X^\times \times Y$ and note that $Z$ projects onto both $X^\times$ and $Y$. We claim that for every 
$P \in X^\times$ the section
\begin{equation}
Z[P] := \{y \in Y\mid (P, y) \in Z\}
\end{equation}
is a singleton. Assuming the claim for the moment, we deduce that  $Z = {\rm gr}(\beta)$ for some map $\beta: X^\times \to Y$, which is $G$-equivariant by $G$-invariance of $Z$ and satisfies $\beta(P_0) = y_0$ by construction. Since $\beta$ has a closed graph, it is automatically Borel. Conversely, if $\beta': X^\times \to Y$ is any $G$-equivariant Borel map with closed graph satisfying $\beta'(P_0) = y_0$, then ${\rm gr}(\beta') \supset Z = {\rm gr}(\beta)$ and thus $\beta' = \beta$. Thus our claim implies both existence and uniqueness of $\beta$. Moreover, in the cocompact case both $X^\times$ and $Y$ are compact, hence $\beta$ is automatically continuous by the closed graph theorem.

To establish the claim, consider first $P \in \mathcal T$ and let $y \in Z[P]$. By definition this means that there exist $g_n \in G$ such that
\[
g_n.(P_0, y_0) \to (P,y).
\]
By Lemma \ref{LemmahP} we can find a subsequence $(g_{n_i})$ of $(g_n)$ and $s_i \in V_i$, $\gamma_i \in \Gamma_G$ such that
\[
y = \lim_{i \to \infty} g_{n_i}.y_0 = \lim_{i \to \infty} (s_i\gamma_i, e)\Gamma = \lim_{i \to \infty} (s_i, \tau(\gamma_i)^{-1})\Gamma = (e, h_P)\Gamma.
\]
Thus $Z[P] = \{ (e, h_P)\Gamma\}$ is a singleton.

Now let $P \in X^\times$ be arbitrary. Since $P \neq \emptyset$ we can pick $p \in P$. By Proposition \ref{PropHull} we have $p^{-1}P \in \mathcal T$, hence if $\{y_1, y_2\} \in Z[P]$, then by $G$-invariance of $Z$ we have
\[
\{p^{-1}y_1, p^{-1}y_2\} \subset Z[p^{-1}P],
\]
and thus  $p^{-1}y_1 = p^{-1}y_2$ by the previous argument. This implies $y_1 = y_2$ and finishes the proof of the claim and shows that $Z = {\rm gr}(\beta)$. For $P \in \mathcal T$ we have also established that
\begin{equation}\label{betaTransversal}
\beta(P) = (e, h_P)\Gamma.
\end{equation}
To show (iii) we consider $P \in X^{\rm ns} \cap \mathcal T$ and, using Lemma \ref{LemmahP}, pick a sequence $(\gamma_i)$ in $\Gamma_G$ such that $\gamma_i P_0 \to P$, $\tau(\gamma_i) \to h_P^{-1}$ and for all $j \geq i$
\begin{equation}\label{AlmostOneOneMain}
P \cap K_i = \gamma_j P_0 \cap K_i.
\end{equation}
Now fix $i \in \mathbb N$ and consider the finite sets $F:= (K_i \times h_P^{-1}W_0) \cap \Gamma$ and $F_j := (K_i \times \tau(\gamma_j)W_0) \cap \Gamma$. Since $P \in X^{\rm ns}\cap \mathcal T$ and $\beta(P) = (e, h_P)\Gamma$, the window $h_P^{-1}W_0$ is $\Gamma$-regular, and every $\Gamma_H$-translate of $W_0$ is regular as well. Since $\tau(\gamma_j) \to h_P^{-1}$ we can thus apply Lemma \ref{LemmaRegularityConvergence}.(ii) to find $j \geq i$ such that $F = F_j$. For such $j$ we can then apply \eqref{AlmostOneOneMain} to obtain
\[
P \cap K_i  = \gamma_j P_0 \cap K_i = \pi_G(F_j) = \pi_G(F) = \tau^{-1}(h_P^{-1}W_0) \cap K_i.
\]
Since $i$ was arbitrary, this implies $P  = \tau^{-1}(h_P^{-1}W_0)$. This finishes the proof of (iii) and shows that $\beta$ is injective on $X^{\rm ns} \cap \mathcal T$.

We now establish (ii). The inclusion $\mathcal T \subset \beta^{-1}((\{e\} \times H)\Gamma)$ has already been established in \eqref{betaTransversal}. Conversely assume that $P \in X^\times$ with $\beta(P)  = (e, h)\Gamma$ for some $h \in H$. By Proposition \ref{PropHull} we have $p^{-1}P \in \mathcal T$ for every $p \in P$, hence there exists $h_p \in H$ such that $\beta(p^{-1}P) = (e, h_p)\Gamma$. It then follows from $G$-equivariance of $\beta$ that
\[
 (e, h_p)\Gamma = \beta(p^{-1}P) = p^{-1}\beta(P) = (p^{-1}, h)\Gamma,
\]
hence $p \in \Gamma_G$. Since $p \in P$ was arbitrary this implies $P \in \mathcal T$ and finishes the proof of (ii).

Concerning (i), assume that $P_1, P_2 \in X^{\rm ns}$ satisfy $\beta(P_1) = \beta(P_2) = (g,h)\Gamma$ for some $g \in G$, $h \in H$. Then
\[
\beta(g^{-1}P_1) = \beta(g^{-1}P_2) = (e, h)\Gamma,
\]
and hence $\{g^{-1}P_1, g^{-1}P_2\} \subset X^{\rm ns} \cap \mathcal T$ by (ii). Since $\beta$ is injective on $X^{\rm ns} \cap \mathcal T$ we deduce that $g^{-1}P_1 = g^{-1}P_2$ and hence $P_1 = P_2$. This proves (i) and finishes the proof.
\end{proof}
It remains to prove the lemmas.

\begin{proof}[Proof of Lemma \ref{LemmaWindowMagic}] (i) Given $n \in \mathbb N$ we define a subset $W_{0, n} \subset W_0^o$ by
\[W_{0,n}:= \{w \in W^o_0 \mid {V_n^{-1}}w \subseteq W_0^o\}\quad (n \in \mathbb N).\]
We first observe that
\begin{equation}\label{inclaux}
W_{0,1} \subset W_{0,2} \subset \dots \subset W^o_0 \quad \text{and} \quad W_{0,n}\subset \bigcap_{v \in V_n}vW_0.
\end{equation}
Indeed, the first statement follows from the fact that the $V_n$ are descreasing, and if $v \in V_n$ and $w \in W_{0,n}$ then $v^{-1}w \in W_0^o \subseteq W_0$, which implies $w \in vW_0$ and thus establishes \eqref{inclaux}.

Now define $M_n := hW_{0,n} \setminus h'\overline{V_n}W_0$. We claim that there exists $n \in \mathbb N$ such that $M_n$ contains a non-empty open set $U$. Assuming the claim for the moment, let us finish the proof. We can find $k_0$ such that for all $k \geq k_0$ we have $h_k \in hV_n$ and $h_k' \in h'{V_n}$. Then by \eqref{inclaux} we have
\[
h_kW_0 \setminus h_k'W_0 \supset \bigcap_{v \in V_n} hvW_0 \setminus h'V_nW_0 \supset h \bigcap_{v \in V_n}vW_0\setminus  h'\overline{V_n}W_0 \supset M_n \supset U,
\]
i.e. $h_kW_0 \setminus h_k'W_0 \supset U$, which is the statement of the lemma. It thus remains to establish the claim.

Firstly, since $W_0$ has trivial stabilizer we have $hW_0 \setminus h'W_0 \neq \emptyset$. Secondly, since $h'W_0$ is compact and $hW_0^o$ is dense in $hW_0$, $hW_0^o \setminus h'W_0 \subset hW_0 \setminus h'W_0$ is dense, and in particular $hW_0^o \setminus h'W_0 \neq \emptyset$. Thirdly, $hW_0^o \setminus h'W_0$ is open, and thus $m_H(hW_0^o \setminus h'W_0) > 0$. From regularity of the Haar measure we thus deduce that
\begin{equation}
\exists m \in \mathbb N: m_H(hW_0^o \setminus h'\overline{V_{m}}W_0) > 0.
\end{equation}
We fix such an $m$ once and for all and observe that $A := hW_0^o \setminus h'\overline{V_{m}}W_0 \neq \emptyset$.

Since the set $A$ is open and non-empty it contains a a basic open set of the form $U_{n, w} := hV_{n}^{-1}h^{-1}V_{n}w$ for some $n \in \mathbb N$ and $w \in G$. We may assume that $n \geq m$. Then we claim that 
\begin{equation}\label{FinalClaimTopol}
U := V_n w \subset M_n.
\end{equation}
Since $U$ is open this will finish the proof. From the inclusion $U_{n,w} \subset A$ we deduce two things: Firstly, $U_{n,w} \subset hW_0^o$, i.e. $V_{n}^{-1}h^{-1}V_{n}w \subset W_0^o$. By the very definition of $W_{0, n}$ this means that $h^{-1}V_{n}w \subset W_{0, n}$, and hence
\begin{equation}\label{WindowLemEq2}
U = V_{n}w \subset hW_{0, n}.
\end{equation}
Secondly, since $e \in V_n^{-1}$ we have $U = V_nw \subset U_{n,w}$ and thus 
\begin{equation}\label{WindowLemEq1}
U \cap h'\overline{V_{m}}W_0 \subset U_{n,w}\cap h'\overline{V_{m}}W_0\subset A \cap h'\overline{V_{m}}W_0 = \emptyset.
\end{equation}
Combining \eqref{WindowLemEq1} and \eqref{WindowLemEq2} and using that $n \geq m$ and hence $\overline{V_n} \subset \overline{V_m}$
we obtain
\[
U \subset hW_{0, n} \setminus h'\overline{V_{m}}W_0 \subset hW_{0, n} \setminus h'\overline{V_{n}}W_0 = M_n.
\]
This establishes \eqref{FinalClaimTopol} and finishes the proof.

(ii) Let $K \subset G$ be a compact set. Since $h W_0$ and $h_n W_0$ are $\Gamma$-regular for every $n$,
we have
\begin{equation}
\label{firsteq}
(K \times h W_0) \cap \Gamma 
= 
\big( (K \times (h W_0 \setminus h_n W_0^o)) \cap \Gamma \big)
\cup
\big( (K \times (h W_0 \cap h_n W_0)) \cap \Gamma \big)
\end{equation}
and
\begin{equation}
\label{secondeq}
(K \times h_n W_0) \cap \Gamma 
= 
\big( (K \times (h_n W_0 \setminus h W_0^o)) \cap \Gamma \big)
\cup
\big( (K \times (h_n W_0 \cap h W_0)) \cap \Gamma \big).
\end{equation}
Hence, in order to show that 
\begin{equation}
\label{wishedfor}
(K \times h W_0) \cap \Gamma = (K \times h_n W_0) \cap \Gamma
\end{equation}
for large enough $n$, it suffices to show that for large enough $n$, we have
\[
\big(K \times (h W_0 \setminus h_{n} W_0^o)\big) \cap \Gamma 
= 
\big( K \times (h_{n} W_0 \setminus h W_0^o)\big) \cap \Gamma 
= 
\emptyset.
\]
Since $W_0$ is compact and $h_n \ra h$, there is a compact set $L \subset H$ such that
$h W_0 \setminus h_{n} W_0^o \subset L$ and $h_{n} W_0 \setminus h W_0^o \subset L$
for all $n$. Since $\Gamma$ is discrete, this shows that the sets 
\[
A_n = \big( K \times (h W_0 \setminus h_{n} W_0^o)\big) \cap \Gamma
\qand
B_n = \big( K \times (h_{n} W_0 \setminus h W_0^o)\big) \cap \Gamma
\]
vary inside the set of all subsets of the \emph{finite} set $T = (K \times L) \cap \Gamma$. 
In particular, for every sub-sequence of $(h_n)$, there is a further sub-sequence $(h_{n_j})$ 
such that the sequences $(A_{n_j})$ and $(B_{n_j})$ are constant. On the other hand, one readily 
verifies that
\[
\bigcap_j (h W_0 \setminus h_{n_j} W_0^o) \subset h \partial W_0
\qand 
\bigcap_j (h_{n_j} W_0 \setminus h W_0^o) \subset h \partial W_0,
\]
for every sub-sequence $(h_{n_j})$, and thus, since $h W_0$ is $\Gamma$-regular, we conclude that 
\[
\bigcap_j A_{n_j} = \emptyset 
\qand 
\bigcap_j B_{n_j} = \emptyset.
\]
We conclude that every sub-sequence of $(h_n)$ admits a further sub-sequence $(h_{n_j})$ such that 
\[
\big(K \times (h W_0 \setminus h_{n_j} W_0^o)\big) \cap \Gamma 
= 
\big( K \times (h_{n_j} W_0 \setminus h W_0^o)\big) \cap \Gamma 
= 
\emptyset.
\]
for all $j$. We claim that this finishes the proof. Indeed, if \eqref{wishedfor} were to fail for infinitely many $n$, then \eqref{firsteq} and \eqref{secondeq} would tell us that we can find a sub-sequence $(h_{n_j})$ such that either
\[
 \big(K \times (h W_0 \setminus h_{n_j} W_0^o)\big) \cap \Gamma 
\neq \emptyset \qor  
\big( K \times (h_{n_j} W_0 \setminus h W_0^o)\big) \cap \Gamma 
\neq 
\emptyset,
\]
for all $j$, which contradicts what we have just proved. 
\end{proof}

\begin{proof}[Proof of Lemma \ref{LemmahP}] Since $g_nP_0 \to P$ we can find for every $i \in \mathbb N$ some $n_i \in \mathbb N$ and $t_i \in V_i$ such that
\begin{equation}\label{OrbitClosureCondition}
t_ig_{n_i}P_0 \cap K_i = P \cap K_i.
\end{equation}
If we define $\gamma_i:= t_ig_{n_i}$ then $\gamma_iP_0 \cap K_i = P \cap K_i$. We deduce that for all $j \geq i$ we have
\[
\gamma_jP_0 \cap K_i = (\gamma_j P_0 \cap K_j) \cap K_i = P \cap K_j \cap K_i = P \cap K_i.
\]
Since $P \cap K_i  \neq \emptyset$, we may assume by passing to a further subsequence that $\gamma_iP_0 \cap K_i \neq \emptyset$ for all $i \in \mathbb N$. Then for every $i \in \mathbb N$ there exists $p_0 \in P_0 \subset \Gamma_G$ such that $\gamma_ip_0 \in \gamma_iP_0 \cap K_i =P \cap K_i$. Since $P \in \mathcal T$ we deduce that $\gamma_ip_0 \in \Gamma_G$ and thus $\gamma_i \in \Gamma_G$. If we now set 
$s_i := t_i^{-1}$, then $s_i \in V_i^{-1} = V_i$ and $g_{n_i} = s_i\gamma_i$. For this choice of $\gamma_i$ and $s_i$ we have thus established (i) and (iii). Also note that $s_i \in V_i$ implies $s_i \to e$.

We next claim that the set $\{\tau(\gamma_i)\}$ is pre-compact. Suppose otherwise for contradiction. Then for every $i \in \mathbb N$ there exists $j > i$ such that
\[
\tau(\gamma_i)W_0 \cap \tau(\gamma_j)W_0 = \emptyset,
\]
and consequently
\[
\pi_G\left([(K_i \times \tau(\gamma_i)W_0) \cap \Gamma] \cap [(K_j \times \tau(\gamma_j)W_0) \cap \Gamma]\right) = \emptyset.
\]
Since $\pi_G|_\Gamma$ is injective, this can be rewritten as
\[
\pi_G\left(K_i \times \tau(\gamma_i)W_0) \cap \Gamma\right) \cap \pi_G\left(K_j \times \tau(\gamma_j)W_0) \cap \Gamma\right) = \emptyset,
\]
or equivalently
\[
(\gamma_iP_0 \cap K_i) \cap (\gamma_jP_0 \cap K_j) = \emptyset. 
\]
This, however, contradicts (iii), and establishes our claim.

In order to establish (ii) and thereby to finish the proof of the lemma it remains to show that every convergent subsequence of $\tau(\gamma_i)$ converges to the same limit $h_P^{-1}$, which is independent of the sequence $g_n$. We argue again by contradiction and assume otherwise. Then there exist cofinal subsets $I, I' \subset \mathbb N$ and sequences $(\gamma_i)_{i \in I}$, $(\gamma'_{i'})_{i' \in I'}$ satisfying (iii) such that $h_i := \tau(\gamma_i)$ and $h_{i'}' := \tau(\gamma'_{i'})$ converge to different elements of $H$. By Lemma \ref{LemmaWindowMagic}.(i) we then find an open set $U$ such that for all sufficiently large $i, i'$,
\[
U \subset \tau(\gamma_i)W_0 \setminus \tau(\gamma'_{i'})W_0.
\]
Since $\Gamma_H$ is dense in $H$ we have $\Gamma_H \cap U \neq \emptyset$ and thus $(G\times U) \cap \Gamma\neq \emptyset$. For sufficiently large $j$ we thus get $(K_{j} \times U) \cap \Gamma \neq \emptyset$, hence
\[
\emptyset \neq (K_j\times U) \cap \Gamma \subset (K_j \times  \tau(\gamma_i)W_0) \cap \Gamma \setminus [(K_j \times \tau(\gamma'_{i'})W_0) \cap \Gamma].
\]
Applying $\pi_G$, which is injective on $\Gamma$, we obtain
\[
\emptyset \neq (\gamma_iP_0 \cap K_j) \setminus  (\gamma'_{i'}P_0 \cap K_j).
\]
However, we may assume that $i$ and $i'$ are larger than $j$. Then (iii) yields
\[
\gamma_iP_0 \cap K_j = P \cap K_j = \gamma'_{i'}P_0 \cap K_{j},
\]
which is a contradiction.
\end{proof}

\section{The auto-correlation of a model set}\label{SecAutocorrelation}

\subsection{The periodization map of an FLC set} 
We consider the following setting: $G$ is a lcsc group and $P_0 \subset G$ is a subset of finite local complexity. We denote by $X := X_{P_0}$ and $X^\times := X^\times_{P_0}$ the hull and the punctured hull of $P_0$ respectively. While we have in mind the case of a (not-necessarily uniform) model set, large parts of the theory can be carried out in larger generality.

\begin{proposition}[Existence of the parametrization map]\label{PropSiegel} There is a well-defined $G$-equivariant map 
\[ \mathcal P: C_c(G) \to C(X), \quad
\mathcal Pf(P) := \sum_{g \in P} f(g).
\]
\end{proposition}
Note that the sum definining $\mathcal Pf(P)$ is finite for every $f \in C_c(G)$ and $P \in X$. Indeed, every such $P$ is locally finite by Proposition \ref{PropHull} and ${\rm supp}(f)$ is compact by assumption. Since equivariance of $\mathcal P$ is obvious, the non-trivial statement in the proposition is continuity of $\mathcal Pf$ in the Chabauty--Fell topology. This is proved in \cite[Prop. 5.1]{BH}.

Note that if $P_0$ is not relatively dense and thus $\emptyset \in X$, then $\mathcal Pf(\emptyset) = 0$ for all $f \in C_c(G)$. Thus the image of $\mathcal P$ is contained in $C_0(X^\times) \subset C_b(X^\times)$. If $P_0$ is relatively dense, then anyway $C(X) = C(X^\times) = C_0(X^\times)$, and thus we can consider $\mathcal P$ as a map into $C_0(X^\times)$ in either case.

\begin{definition} The map $\mathcal P: C_c(G) \to C_0(X^\times)$ is called the \emph{periodization map} of $P_0$.
\end{definition}

The issue of continuity of the periodization map is a subtle one. The periodization map is in general not continuous when $C_c(G)$ and  $C_0(X^\times)$ are equipped with the respective topologies of uniform convergence on compacta. However, it is continuous with respect to the natural Fr\'echet topology on $C_c(G)$ and the topology of uniform convergence on $C_0(X^\times)$. Explicitly, this means the following:
\begin{lemma}[Continuity of the periodization map, {\cite[Prop. 5.4]{BH}}]\label{PeriodizationContinuous} For every compact subset $K \subset G$ there exists $C_K > 0$ such that if $f \in C_c(G)$ with ${\rm supp}(f) \subset K$, then
\begin{equation}
\|\mathcal Pf\|_\infty \quad \leq \quad C_K \cdot \|f\|_\infty.
\end{equation}
\end{lemma}

\subsection{Periodization of measures} The continuity property of the periodization map as expressed by Lemma \ref{PeriodizationContinuous} implies that if $\nu$ is a probability measure on $X^\times$, then we can define a Radon measure $\mathcal P^*\nu$ on $G$ by $\mathcal P^*\nu(f) := \nu(\mathcal Pf)$, where we think of Radon measures as linear functionals on $C_c(G)$. More generally we define for every $n \geq 1$ a Radon measure ${\eta}_\nu^{(n)}$ on $G^n$ by
\[
{\eta}_\nu^{(n)}(f_1\otimes \dots \otimes f_n) := \int_{X^\times} \mathcal Pf_1 \cdots \mathcal Pf_n \, d\nu \quad(f_1, \dots, f_n \in C_c(G)),
\]
which we call the \emph{$n$-th  correlation measure} of $\nu$. Note that it follows from equivariance of the periodization map, that if $\nu$ is a $G$-invariant measure on $X^\times$, then the correlation measures ${\eta}_\nu^{(n)}$ are invariant under the action of the diagonal subgroup $\Delta(G)$ on $G^n$ by left-multiplication.
\begin{proposition}\label{StoneWeierstrass} Every probability measure on $X^\times$ is determined by its correlation measures.
\end{proposition}
\begin{proof} By \cite[Prop. 5.3]{BH} the image of the periodization map separates point in $X^\times$. If $X^\times = X$, then the proposition thus follows from the Stone--Weierstra\ss\ theorem. If $X^\times$ is non-compact, then its one-point compactification is given by $X$. In this case, the set $\mathcal P(C_c(G)) \cup \{1_X\}$ separates points of $X$. Applying the Stone--Weierstra\ss\ theorem to this set, we deduce that the algebra generated by $\mathcal P(C_c(G))$ and $1_X$ is dense in $C(X)$, which implies that the algebra generated by $\mathcal P(C_c(G))$ is dense in $C_0(X^\times)$ and finishes the proof.
\end{proof}

\subsection{Auto-correlation measures of an invariant measure} We recall that a lcsc group $G$ is \emph{unimodular} if some left-Haar measure on $G$ is also right-invariant. Every lcsc group containing a lattice, or more generally a model set, is necessarily unimodular. To simplify our exposition we are going to make the
\begin{convention} From now on, all lcsc groups are assumed to be unimodular.
\end{convention}
Given a unimodular lcsc group $H$ and a closed unimodular subgroup $K < H$ we denote by 
 \[\mathcal P_K: C_c(H) \to C_c(K \backslash H), \quad \mathcal P_K(f)(Kh) := \int_K f(kh)dm_K(k)\]
the periodization map. In order to define the auto-correlation measures of an invariant measure we will make use of the following general lemma. 
\begin{lemma} Let $H$ be a unimodular lcsc group and $K < H$ a closed unimodular subgroup. Then for every $K$-invariant Radon measure $\eta$ on $H$ there is a unique Radon measure $\overline{\eta}$ on $K \backslash H$ such that
\[
\eta(f) = \overline{\eta}(\mathcal P_K f) \quad (f \in C_c(H)).
\]
\end{lemma}
\begin{proof} Uniqueness of $\overline{\eta}$ follows from surjectivity of the periodization map (see e.g. \cite[Lemma 1.1.1]{Raghunathan}).  Given $\bar f \in C_c(K \backslash H)$, we would like to define $\overline{\eta}(\bar f) := \eta(f)$, where $f$ is a pre-image $\bar f$ under $\mathcal P_K$. To show that this is well-defined we need to show that if $f \in C_c(K)$ satisfies $\mathcal P_Kf = 0$, then $\eta(f) = 0$. Assume for contradiction that $\mathcal P_Kf = 0$, but $\eta(f) \neq 0$. Choose $\rho_n \in C_c(H)$ such that $\mathcal P_K(\rho_n) \to 1$ and observe that
\[
\int_H f(h) \cdot \left(\int_K \rho_n(kh)dm_K(k)\right)d\eta(h) \quad \to\quad \eta(f) \quad \neq \quad 0.
\]
On the other hand, using $K$-invariance of $\eta$ and unimodularity of $K$ we obtain for all $n \in \mathbb N$,
\begin{eqnarray*}
&& \int_H f(h) \cdot \left(\int_K \rho_n(kh)dm_K(k)\right)d\eta(h) \quad =\quad \int_H \left(\int_K f(k^{-1}h) dm_K(k)\right) \rho_n(n) d\eta(h)\\
&=& \int_H \left(\int_K f(kh) dm_K(k)\right) \rho_n(n) d\eta(h) \quad = \quad \int_H \mathcal P_K(f)(Kh)\cdot \rho_n(h) d\eta(h) \quad = \quad 0.
\end{eqnarray*}
This contradiction finishes the proof.
\end{proof}
%


If $\nu$ is a $G$-invariant probability measure on $X^\times$, then its $n$th correlation measure $\eta_\nu^{(n)}$ is a $\Delta(G)$-invariant Radon measure on $G^n$, hence descends to Radon measures $\overline{\eta}_\nu^{(n-1)}$ on the quotient $\Delta(G) \backslash (G \times \dots \times G)$, which we can identify with $G^{n-1}$ via the isomorphism
\[
\Delta(G) \backslash (G \times \dots \times G) \to G^{n-1}, \quad (g_1, \dots, g_n) \mapsto (g_1^{-1}g_2, \dots, g_1^{-1}g_n).
\]
\begin{definition} If $\nu$ is a $G$-invariant probability measure on $X^\times$, then the Radon measure $\overline{\eta}_\nu^{(n)}$ on $G^{n}$ is called the \emph{$n$-th auto-correlation measure} of $\nu$. In particular, the Radon measure $\eta_\nu := \overline{\eta}_\nu^{(1)}$ on $G$ is called the \emph{auto-correlation measure} of $\nu$.
\end{definition}
It follows from Proposition \ref{StoneWeierstrass} that a $G$-invariant probability measure $\nu$ on $X^\times$ is uniquely determined by its auto-correlation measures.
The first auto-correlation measure has a special property not shared by the higher auto-correlation measures.
\begin{proposition}[Positive-definiteness of the auto-correlation]\label{PropAutocor} The auto-correlation $\eta_\nu$ of an invariant probability measure $\nu$ on $X^\times$ is a positive-definite Radon measure on $G$. In fact, it is the unique positive-definite Radon measure on $G$ such that
\begin{equation}\label{Autocor}
\eta_\nu\left(f^* \ast f\right) = \|\mathcal P f\|_{L^2(X^\times, \nu)}^2 \quad (f \in C_c(G)).
\end{equation}
\end{proposition}
\begin{proof} Note that if $f_1, f_2 \in C_c(G)$, then
\[
\eta_{\nu}^{(2)}(\overline{f_1} \otimes f_2)  \quad=\quad \int_{X^\times} \overline{\mathcal P f_1} \cdot \mathcal Pf_2 \; d\nu \quad=\quad \langle \mathcal Pf_1, \mathcal Pf_2\rangle_{L^2(X^\times, \nu)}.
\] 
Next observe that for all $g, h \in G$ we have
\begin{eqnarray*}
f_{12}(g^{-1}h) 
&:=& 
\int_{G} (\overline{f_1} \otimes f_2)(rg,rh) \, dm_G(r)\\
&=& 
\int_G \overline{f_1}(r)f_2(rg^{-1}h) \, dm_G(r) \\
&=& 
\int_G f_1^*(r) f_2(r^{-1}g^{-1}h) \, dm_G(r) 
= (f_1^* * f_2)(g^{-1}h).
\end{eqnarray*}
It follows that
\[
\eta_\nu\left(f_1^* * f_2\right) \quad=\quad  \eta_\nu(f_{12}) \quad=\quad \eta_{\nu}^{(2)}(\overline{f_1} \otimes f_2) \quad=\quad  \langle \mathcal Pf_1, \mathcal Pf_2\rangle_{L^2(X^\times, \nu)}.
\]
Specializing to $f_1 = f_2 =: f$ we obtain \eqref{Autocor}. In particular, $\eta_\nu\left(f^* \ast f\right) \geq 0$ for all $f \in C_c(G)$, showing that $\eta_\nu$ is positive-definite. Finally, $\eta_\nu$ is uniquely determined by \eqref{Autocor}, since $\{f^* \ast f \mid f \in C_c(G)\}$ is dense in $C_c(G)$.
\end{proof}
For the rest of this article we will focus almost exclusively on the (first) auto-correlation measure and not consider the higher auto-correlation measures. We will, however, briefly comment on the $0$th auto-correlation measure in Subsection \ref{SecSiegel}.

\begin{remark} If $P_0$ is an arbitrary FLC subset of a lcsc group $G$, then in general there may not exist any $G$-invariant probability measure on $X^\times$. There are two notable exceptions. Firstly, if $G$ is amenable and $P_0\subset G$ is a Delone set, then $X^\times = X$ is compact, and thus there will always exist a $G$-invariant probability measure on $X^\times$. Secondly, if $P_0$ happens to be a regular model set, then a $G$-invariant probability measure on $X^\times$ exists by Theorem \ref{ThmUE}, and there is in fact a unique such measure.
\end{remark}
\begin{definition} Assume that $P_0 \subset G$ is a FLC set whose punctured hull $X^\times$ admits a unique $G$-invariant probability measure $\nu$. Then the auto-correlation measure $\eta_{P_0} := \eta_\nu$ is called the \emph{auto-correlation measure} of $P_0$.
\end{definition}

\subsection{A formula for the auto-correlation of a model set}
We now apply the theory developed so far to the case of model sets. For the rest of this section we fix a regular model set $P_0 = P_0(G, H, \Gamma, W_0)$ in a lcsc group $G$. 
As before we denote by $X^\times$ the punctured hull of $P_0$, by $\nu$ the unique $G$-invariant probability measure on $X^\times$ and by $Y := (G \times H)/\Gamma$ the associated parameter space. By Theorem \ref{ThmParametrizationMap} we have a parametrization map $\beta: X^\times \to Y$ which satisfies $\beta_\ast \nu = m_Y$, where $m_Y$ denotes the invariant probability measure on $Y$. Recall that the auto-correlation measure $\eta_{P_0}$  of $P_0$  was defined as the auto-correlation measure $\eta_{\nu}$ of $\nu$.

Given a bounded Riemann-integrable function $F : G\times H \to \mathbb R$ with compact support we denote by $\mathcal P_\Gamma F$ the $\Gamma$-{\em periodization} of $F$, i.e. the function
\[
\mathcal P_\Gamma F: Y \to \mathbb R, \quad \mathcal P_\Gamma F((g,h)\Gamma) = \sum_{\gamma \in \Gamma} F((g,h)\gamma).
\]
Then we have the following formula for the auto-correlation of $P_0$.
\begin{theorem}[Auto-correlation of model sets]\label{AutocorrelationFormula} 
Let $\eta = \eta_{P_0}$ be the auto-correlation measure of $P_0$. Then $\eta$ is the unique Radon measure on $G$ satisfying
\begin{equation}
\eta(f^\ast \ast f) = \|\mathcal P_\Gamma(f \otimes \chi_{W_0})\|^2_{L^2(Y)} \quad(f \in C_c(G)).
\end{equation}
\end{theorem}
\noindent In view of \eqref{Autocor} this is an immediate consequence of the following lemma and the fact that $\nu(X^{\rm ns})=m_Y(Y^{\rm ns}) = 1$.
\begin{lemma}\label{LemmaSiegelFormula}\label{PeriodizationLemma} If $P \in X^{\rm ns}$ and $f \in C_c(G)$, then
\[
\mathcal P f(P) = \mathcal P_\Gamma(f \otimes \chi_{W_0})(\beta(P)).
\]
\end{lemma}
\begin{proof} Let $P \in X^{\rm ns}$ and $p \in P$. By Proposition \ref{PropHull} and $G$-invariance of $X^{\rm ns}$ we then have $P' := p^{-1}P \in \mathcal T \cap X^{\rm ns}$. By Theorem \ref{ThmParametrizationMap} we then have
\[
\beta(P') = (e, h_{P'})\,\Gamma \quad \text{and}\quad \tau^{-1}(h_{P'}^{-1}W_0) = P'.
\]
Now $G$-equivariance of $\beta$ yields 
\[
\beta(P) = \beta(pP') = p\beta(P') = p(e, h_{P'})\, \Gamma = (p, h_{P'})\, \Gamma,
\]
and thus we obtain for every $f \in C_c(G)$,
\begin{eqnarray*}
\mathcal P_{\Gamma}(f\otimes \chi_{W_0})( \beta(P)) &=& \mathcal P_{\Gamma}(f\otimes \chi_{W_0})((p, h_{P'})\, \Gamma) =  \sum_{\gamma \in \Gamma_G} f(p\gamma) \chi_{W_0}(h_{P'}\tau(\gamma))\\
&=& \sum_{\gamma \in \tau^{-1}(h_{P'}^{-1}W_0)} f(p\gamma) = \sum_{\gamma \in P'} f(p\gamma) = \sum_{\gamma \in pP'}f(\gamma) = \sum_{\gamma \in P} f(\gamma)\\
&=& \mathcal Pf(P).
\end{eqnarray*}
\end{proof}
\subsection{The covolume of regular model sets}\label{SecSiegel}
If $P_0$ is a regular model set in $G$, then by definition $P_0$ arises from a cut-and-project scheme $(G, H, \Gamma)$ and window $W_0$, but $H$, $\Gamma$ and $W_0$ are not uniquely determined by $P_0$. In this subsection we explain how to use the $0$the auto-correlation to construct an invariant of $P_0$ from some normalized volume of the window $W_0$. 

Let us fix a reference Haar measure $m_G$ on $G$ once and for all. Given a Haar measure $m_H$ on $H$, we denote by ${\rm covol}(\Gamma)$ the covolume of $\Gamma$ in $G \times H$ with respect to $m_G \otimes m_H$. Then the quotient
\[
{\rm covol}_{m_G}(P_0) := \frac{{\rm covol}(\Gamma)}{m_H(W_0)}
\] 
does not depend on the choice of Haar measure $m_H$ on $H$, and we refer to this quotient as the \emph{covolume} of $P_0$ in $G$ with respect to $m_G$. This is motivated by the fact that if $H$ is the trivial group, then $P_0$ is a lattice in $G$ and ${\rm covol}_{m_G}(P_0)$ is the covolume of $P_0$ in $G$ with respect to $m_G$ in the usual sense. We claim that the covolume depends only on $m_G$ and the regular model set $P_0 = P_0(G, H, \Gamma, W_0)$, but not on the choices of $H$, $\Gamma$ and $W_0$. To establish the claim, let us denote by $\nu$ the unique invariant measure on $X^\times = X^\times_{P_0}$. Since the first correlation measure $\eta^{(1)}_\nu$ of $\nu$ is $G$-invariant, it is a multiple of $m_G$ and we claim:
\begin{proposition}
If $P_0 = P_0(G, H, \Gamma, W_0)$ is a regular model set, then
\[
\eta^{(1)}_\nu(f) = \int_{X^\times} \mathcal P f \; d\nu = \frac 1 {{\rm covol}_{m_G}(P_0)} \cdot \int_G f\; dm_G \quad (f \in C_c(G)). 
\]
In particular, ${\rm covol}_{m_G}(P_0)$ depends only on $P_0$ and $m_G$, but not on $H$, $\Gamma$ and $W_0$.
\end{proposition}
\begin{proof} Replacing $m_H$ by ${\rm covol}(\Gamma)^{-1} \cdot m_H$ we may assume that ${\rm covol}(\Gamma) =1$. Since
$\beta_*\nu = m_Y$, Lemma \ref{LemmaSiegelFormula} then implies that for every $f \in C_c(G)$,
\begin{eqnarray*}
\int_{X^{\times}} \mathcal P f \, d\nu &=& \int_{Y} \mathcal P(f \otimes \chi_{W_0}) dm_Y\\
&=& \int_G \int_H f \otimes \chi_{W_0} dm_G dm_H\\
&=& m_H(W_0) \cdot m_G(f), 
\end{eqnarray*}
which finishes the proof.
\end{proof}

\section{Approximations of the auto-correlation measure}\label{SecApprox}

\subsection{An abstract approximation theorem for the auto-correlation} We return to the general setting where $G$ is a lcsc group and $P_0 \subset G$ is a subset of finite local complexity. We denote by $X := X_{P_0}$ and $X^\times := X^\times_{P_0}$ the hull and the punctured hull of $P_0$ respectively and assume that there exists a $G$-invariant probability measure $\nu$ on $X^\times$. Our goal is to express the auto-correlation measure $\eta_\nu$ of $\nu$ as a limit of finite sums of Dirac measures, similarly to the classical definition of the auto-correlation of a uniform model set in an abelian group.

More precisely, our goal is to find conditions on a sequence $(F_t)$ of subsets of $G$ and a class of functions $\mathcal A \subset C_c(G)$ such that the finite sums
\begin{equation}\label{sigmatf}
\sigma_t(f) := \frac{1}{m_G(F_t)} \sum_{x \in P_0 \cap F_t} \sum_{y \in P_0} f(x^{-1}y)\end{equation}
approximate $\eta_\nu(f)$ for all $f \in \mathcal A$.

To formulate the conditions on the sequence $(F_t)$, we fix an admissible metric on $G$ and denote by $B_\delta$ the open ball of radius $\delta$ around the identity. Given a subset $L \subset G$ and $\delta > 0$ we then denote 
\[
L_\delta = L B_\delta \qand L_{-\delta} = \bigcap_{t \in B_\delta} Lt.
\]
The following definition is a weakening of the notion of an \emph{admissible sequence} from \cite{GorodnikN-10}.

\begin{definition}\label{DefWeaklyAdmissible}
We say that a sequence $(F_t)$ of compact subsets of $G$ is 
\emph{weakly admissible} if each $(F_t)$ has positive Haar measure and there are continuous functions $\alpha, \beta : [0,1) \ra \bR_{+}$ 
with $\alpha(0) = 0$ and $\beta(0) = 0$ such that
\[
(F_t)_\delta \subset F_{t + \alpha(\delta)}
\qand
\sup_s \frac{m_G(F_{s + \delta})}{m_G(F_s)} = 1  + \beta(\delta),
\]
for all $t, \delta > 0$. We shall refer to the pair $(\alpha, \beta)$ as the \emph{parameters} of 
$(F_t)$.
\end{definition}
Concerning the class of functions $\mathcal A$ we are going to assume the following condition.
\begin{definition}\label{DefGeneric}
Let $(F_t)$ be a sequence of compact subsets of $G$ of positive Haar measures. We say that a linear sub-space $\cA \subset C_c(G)$ is \emph{generic} 
with respect to $\nu$ and the sequence $(F_t)$ if
\[
\eta_{\nu}(f_1^* * f_2) 
= 
\lim_{t \to \infty} \frac{1}{m_G(F_t)} \int_{F_t} \overline{\mathcal P f_1(s^{-1}. P_0)} \,\mathcal P f_2(s^{-1}. P_0) \; dm_G(s)
\quad \textrm{for all $f_1, f_2 \in \cA$}.
\]
\end{definition}
Now we can state the desired approximation theorem.

\begin{theorem}[Abstract approximation theorem]
\label{approxthm1}
Suppose that $(F_t)$ is a weakly admissible sequence of compact subsets of $G$ and that $\cA \subset C_c(G)$ is generic with respect to $\nu$ and $(F_t)$.
Then for every $f \in \cA$ the finite sums \eqref{sigmatf} converge to the auto-correlation, i.e.
\[
\eta_{\nu}(f) = \lim_{t \ra \infty} \sigma_t(f).
\]
\end{theorem}
The proof of Theorem \ref{approxthm1} will be given in Subsection \ref{SubsecProofApprox}. We emphasize that no compactness assumption on $X^\times$ is made in Theorem \ref{approxthm1}.

\subsection{An approximation theorem for FLC subsets of amenable groups}
To illustrate how Theorem \ref{approxthm1} can be applied in practice, we consider first the case where $G$ is amenable.
\begin{corollary}[Approximation theorem for amenable groups]\label{approxthmamenable} Assume that $G$ is amenable and let $(F_t)$ be a weakly admissible F\o lner sequence of compact subsets in $G$. If the hull $X^\times$ of a FLC subset $P_0 \subset G$ is compact and uniquely ergodic, then
\[
\eta_{P_0}(f) =\lim_{t \ra \infty} \sigma_t(f)
\quad 
\textrm{for all $f \in C_c(G)$}.
\]
\end{corollary}
\begin{proof} If $G$ is amenable and $X^\times$ is compact and uniquely ergodic, then the pointwise ergodic theorem for uniquely ergodic systems (see e.g. \cite[Thm. 4.10]{EinsiedlerWard} for $G= \Z$; the extension to amenable groups is straight-forward) implies that $C_c(G)$ is generic with respect to any F\o lner sequence. Then Corollary \ref{approxthmamenable} follows from Theorem \ref{approxthm1}. 
\end{proof}

\subsection{An approximation theorem for regular model sets in non-amenable groups}
The goal of this subsection is to illustrate that Theorem \ref{approxthm1} also applies in many non-amenable situations. To this end we are going to establish a general theorem concerning the approximation of the autocorrelation measure of a regular model sets, which contains Theorem \ref{ThmLieGroupsIntro} from the introduction on regular model sets in semisimple Lie groups as a special case.

Throughout this subsection $P_0 = P_0(G, H, \Gamma, W_0)$ denotes a regular model set in a lcsc group $G$ in the sense of Definition \ref{def_regular}. For large parts of this subsection $G$ can be an arbitrary lcsc group, although for the explicit examples below we will assume that $G$ is a semisimple Lie group. We recall from Theorem 
\ref{main1} that the punctured hull $X^{\times}$ of $P_0$ admits a unique $G$-invariant probability measure $\nu$, and by definition $\eta_{P_0}= \eta_{\nu}$. In the sequel we denote by $Y := (G\times H)/\Gamma$ the parameter space of $X^\times$ and by $\beta : X^\times \ra Y$ the parametrization map from Theorem \ref{ThmParametrizationMap}. We also recall from Lemma \ref{PeriodizationLemma} that for $f \in C_c(G)$ we have
\begin{equation}
\label{betaprop}
\beta^* \cP_\Gamma(f \otimes \chi_{W_o}) = \cP f, \quad \textrm{$\nu$-almost everywhere}.
\end{equation}
Our goal is to find conditions on a weakly admissible sequence $(F_t)$ of compact subsets of $G$ which guarantee that $C_c(G)$ is generic with respect to $\nu$ and $(F_t)$, for then Theorem \ref{approxthm1} will imply that
\[
\eta_{P_0}(f) = \lim_{t \ra \infty} \sigma_t(f) \quad \text{for all }f\in C_c(G).
\]
We begin by introducing some notation. Given a sequence $(F_t)$ of compact sets in $G$ with positive measures, we define a
sequence $(\beta_t)$ of probability measures on $G$ by
\begin{equation}
\label{defbetat}
\beta_t(\phi) = \frac{1}{m_G(F_t)} \int_{F_t} \phi(g) \, dm_G(g) \quad \textrm{for $\phi \in C_c(G)$}.
\end{equation}
\begin{definition}[(Almost) everywhere goodness]
We say that a sequence $(F_t)$ of compact subsets of $G$ with positive measures is 
\begin{itemize}
\item \emph{everywhere good} for $(Y,m_Y)$ if
\[
(\beta_t * f)(y) \ra \int_Y f \, dm_Y, \quad \textrm{for all $f \in C_c(Y)$ and \emph{every} $y \in Y$}.
\]
By a straightforward approximation argument, the same will also hold for all compactly supported 
$m_Y$-Riemann integrable functions on $Y$. \\

\item \emph{almost everywhere good} for $(Y,m_Y)$ if 
\[
(\beta_t * f)(y) \ra \int_Y f \, dm_Y, \quad \textrm{for all $f \in L^2(Y,m_Y)$ and \emph{$m_Y$-almost every} $y \in Y$}.
\]
\end{itemize}
\end{definition}
\begin{theorem}[Approximation theorem for the auto-correlation of a general regular model set]\label{mainapprox} Let $P_0$ be a regular model set in a lcsc group $G$ and let $(F_t)$ be a weakly admissible sequence  of compact and symmetric subsets of $G$ which are almost everywhere good for $(Y, m_Y)$. Then
\[
\eta_{P_0}(f) =\lim_{t \ra \infty} \sigma_t(f)
\quad 
\textrm{for all $f \in C_c(G)$}.
\]
\end{theorem} 
\begin{example}[Weakly admissible sequences in semisimple Lie groups, after Gorodnik--Nevo \cite{GorodnikN-10}] If $G$ is a connected semisimple real 
Lie group, and $(F_t)$ is the sequence of compact and symmetric sets in $G$ as defined in Theorem \ref{ThmLieGroupsIntro}, then
\begin{itemize}
\item $(F_t)$ is almost everywhere good for $(Y,m_Y)$ (in fact, almost everywhere good for any ergodic $G$-space).
\item $(F_t)$ is weakly admissible (and thus in particular quasi-uniform, see below).
\end{itemize}
Both results are established in Sections 3 and 4 of \cite{GorodnikN-10}. (In fact, this paper discusses the notions above and many variants thereof in great details,
and many examples are worked out. )
\end{example}
Note that Theorem \ref{ThmLieGroupsIntro} is an immediate consequence of Theorem \ref{mainapprox} and the previous example. We now turn to the proof of Theorem \ref{mainapprox}. We will need the following notion.
\begin{definition}[Quasi-uniformity]
We say that a sequence $(F_t)$ of compact subsets of $G$ with positive measures is 
\emph{quasi-uniform} if
\begin{itemize}
\item For every $\eps > 0$, there is an open neighborhood $\mathcal{O}$ of $e$ in $G$
such that for all sufficiently large $t$, we have $\mathcal{O} F_t \subset F_{t+\eps}$.
\item For every $\delta > 0$, there exists $\eps > 0$ such that for all sufficiently large  $t$,
\[
m_G(F_{t+\eps}) \leq (1 + \delta)m_G(F_t).
\]
\end{itemize}
\end{definition}
We note that every weakly admissible sequence of compact and symmetric sets is automatically quasi-uniform. In view of Theorem \ref{approxthm1} above it thus remains to show only that if $(F_t)$ is a weakly admissible sequence of compact subsets of $G$ which is quasi-uniform and almost everywhere good for $(Y,m_Y)$, then $C_c(G)$ is generic with respect to $\nu$ and $(F_t)$. We will break the proof into the following two lemmas.
\begin{lemma}
\label{everywhere}
Suppose that $(F_t)$ is a quasi-uniform sequence of compact subsets in $G$ of positive measures
which is almost everywhere good for $(Y,m_Y)$. Then $(F_t)$ is everywhere good for $(Y,m_Y)$.
\end{lemma}

\begin{proof}
We pick $f \in C_c(G)$. By Theorem 5.22 in \cite{GorodnikN-10}, there exists a \emph{$G$-invariant} and $m_Y$-conull subset $Y^f \subset Y$ such that
\[
(\beta_t * f)(y) \ra \int_Y f \, dm_Y, \quad \textrm{for all $y \in Y^f$}.
\]
Given a \emph{compact} subset $K \subset Y$, we can find a countable dense subset $D \subset C(K)$. A 
straightforward approximation argument shows that 
\[
Y_K := \bigcap_{f \in D} Y^f = \big\{ y \in Y \, : \, (\beta_t * f)(y) \ra \int_Y f \, dm_Y, \enskip \textrm{for all $f \in C(K)$} \big\}.
\]
We see that $Y_K$ is again a $G$-invariant and $m_Y$-conull set, and we shall prove that $Y_K = Y$. Since $K$ is arbitrary, this finishes the proof of Lemma \ref{everywhere}. \\

To do this, let us fix a compact set $K \subset Y$ and an open and pre-compact set $U \subset H$ 
which contains $e$. One readily shows that the irreducibility of $\Gamma$ forces 
\begin{equation}
\label{GUH}
(G \times U) \Gamma = G \times H.
\end{equation}
We set $L = (\{e\} \times U)K$ and note that $L$ is again a compact subset of $Y$. Since $K \subset L$, we clearly have
the inclusion $Y_L \subset Y_K$. Furthermore, if $f \in C(K)$ and $h \in U$, then $f_h(y) := f((e,h)y)$ belongs to $C(L)$, so 
if $y \in Y_L$, then 
\[
(\beta_t * f)((e,h)y) = (\beta_t * f_h)(y) \ra \int_Y f_h \, dm_Y = \int_Y f \, dm_Y.
\]
In other words, if $y \in Y_L$, then $(e,h)y \in Y_K$ for all $h \in U$, or equivalently, 
\[
(\{e\} \times U)Y_L \subset Y_K.
\]
Since $Y_L$ is $m_Y$-conull and $\Gamma$ is countable, we see that $Y'_L := \bigcap_{\gamma \in \Gamma} \gamma Y_L$ is again a $m_Y$-conull set. Furthermore, 
\[
(\{e\} \times U)Y'_L = (\{e\} \times U) \Gamma Y_L' \subset (\{e\} \times U) Y_L \subset Y_K.
\]
Since $Y_K$ is $G$-invariant, we see that
\[
(G \times \{e\}) (\{e\} \times U)Y'_L = (G \times U) \Gamma Y_L' \subset (G \times \{e\})Y_K = Y_K.
\]
By \eqref{GUH}, we conclude that $(G \times U)\Gamma Y' = (G \times H)Y' = Y$,  whence $Y_K = Y$.
\end{proof}

\begin{lemma} 
\label{lemma:genericity}
Suppose that $(F_t)$ is a sequence of compact subsets in $G$ of positive measures
which is everywhere good for $(Y,m_Y)$. Then $C_c(G)$ is generic with respect to 
$\nu$ and $(F_t)$. 
\end{lemma}

\begin{proof}
Given $f_1, f_2 \in C_c(G)$, we define the compactly supported $m_Y$-Riemann integrable function $u : Y \to \mathbb{C}$ by
\begin{eqnarray*}
u(y) :=  \overline{\mathcal P_\Gamma(f_1 \otimes \chi_{W_0})(y)} \mathcal P_\Gamma(f_2 \otimes \chi_{W_0})(y).
\end{eqnarray*}
By \eqref{betaprop} above, we have
\[
\int_Y u \, dm_Y = \int_{X^{\times}} \overline{\cP f_1} \cP f_2 \, d\nu = \langle \cP f_1, \cP f_2 \rangle_{L^2(X^{\times},\nu)}.
\]
Since $(F_t)$ is everywhere good with respect to $m_Y$ and $h$ is $m_Y$-Riemann integrable, we have
\[
\int_Y u \,dm_Y  = \lim_{t \to \infty} \int_G u(g^{-1} y) \, d\beta_t(g), \quad \textrm{for every $y \in Y$}.
\]
If we combine these two observations (with $y = y_o$), we get
\begin{eqnarray*}
\eta_\nu(f_1^* * f_2) &=& \langle \mathcal{P}f_1, \mathcal{P}f_2\rangle_{L^2(X^{\times},\nu)}  = \int_Y u \,dm_Y\\
&=&   \lim_{t \to \infty} \int_G u(g^{-1} y_0) \, d\beta_t(g)\\
&=& \lim_{t \to \infty}  \frac{1}{m_G(F_t)} \; \int_{F_t}  u(s^{-1} y_0) dm_G(s)\\
&=&  \lim_{t \to \infty}  \frac{1}{m_G(F_t)} \; \int_{F_t}   \overline{\mathcal P_\Gamma(f_1 \otimes \chi_{W_0})(s^{-1}.y_0)} \mathcal P_\Gamma(f_2 \otimes \chi_{W_0})(s^{-1}.y_0)dm_G(s)\\
&=& \lim_{t \to \infty} \frac{1}{m_G(F_t)} \int_{F_t} \overline{\mathcal{P} f_1(s^{-1} \cdot P_0)} \mathcal{P} f_2(s^{-1} \cdot P_0) \, dm_G(s).
\end{eqnarray*}
\end{proof}
Now Theorem \ref{mainapprox} follows from Lemma \ref{everywhere}, Lemma \ref{lemma:genericity} and Theorem \ref{approxthm1}.

\subsection{The proof of the abstract approximation theorem}\label{SubsecProofApprox}
The proof of Theorem \ref{approxthm1} is based on the following observations:
\begin{lemma}
\label{approxlemma}
Suppose that $\psi$ is a non-negative left uniformly continuous function on $G$.
Then, for every $\eps > 0$, there exists $\delta > 0$, such that
\[
\sum_{x \in P_0 \cap L_{-\delta}} \psi(x) - \eps | P_0 \cap L_\delta|
\leq 
\int_L (\mathcal P \rho)(s^{-1} P_0) \psi(s) dm_G(s)
\leq 
\sum_{x \in P_0 \cap L_\delta} \psi(x) + \eps | P_0 \cap L_\delta|
\]
for every compact set $L \subset G$ and non-negative $\rho \in C_c(G)$ with $\supp \rho \subset B_\delta$ and $\int_G \rho \, dm_G = 1$.
\end{lemma}

\begin{lemma}
\label{PofiniteinFt}
For every $\delta > 0$, there exists a constant $M_\delta$ such that
\[
M_\delta \geq \frac{|P_0 \cap L_{-\delta} |}{m_G(L)} \quad \textrm{for every compact subset $L \subset G$}.
\]
\end{lemma}
Applying this to weakly admissible sequences we obtain:
\begin{corollary}
\label{corPofiniteinFt}
For every weakly admissible sequence $(F_t)$, we have
\[
\sup_t \frac{|P_0 \cap F_t|}{m_G(F_t)} < \infty.
\]
\end{corollary}
Let us first explain how Lemma \ref{approxlemma} and Corollary \ref{corPofiniteinFt} imply Theorem \ref{approxthm1}.

\begin{proof}[Proof of Theorem \ref{approxthm1}]
Fix $f \in \cA$, which we may assume is non-negative, and note that $\psi(s) = \mathcal P f(s^{-1} \cdot P_0)$ is 
non-negative and left-uniformly continuous on $G$. Let $(F_t)$ be a weakly admissible sequence of subsets 
in $G$ with associated parameters $(\alpha,\beta)$. We wish to prove that 
\[
\eta_{\nu}(f) = \lim_{t \to \infty} \frac{1}{m_G(F_t)} \sum_{x \in P_0 \cap F_t} \psi(x).
\]
Fix $\eps > 0$ and choose $\delta > 0$ as in Lemma \ref{approxlemma} so that for every non-negative continuous function $\rho$ supported on $B_\delta$ with $\int_G \rho \, dm_G = 1$ and every $t \in \R$ we have
\[
\sum_{x \in P_0 \cap (F_t)_{-\delta}} \psi(x) - \eps | P_0 \cap (F_t)_\delta|
\leq 
\int_{F_t} (\mathcal P \rho)(s^{-1} P_0) \psi(s) dm_G(s)
\leq 
\sum_{x \in P_0 \cap (F_t)_\delta} \psi(x) + \eps | P_0 \cap (F_t)_\delta|.
\]
If we define
\[
\Xi_\rho(t) = \frac{1}{m_G(F_t)} \int_{F_t} (\mathcal P \rho)(s^{-1} P_0) \psi(s) dm_G(s),
\]
then, since $\cA$ is assumed to be generic with respect to $\nu$ and $(F_t)$, we have
\[
\eta_{\nu}(\rho^* * f) = \lim_{t \to \infty} \Xi_\rho(t). 
\]
Since $(F_t)$ is weakly admissible, we have for all $ t>0$,
\[
F_{t-\alpha(\delta)} \subset (F_t)_{-\delta}  \qand (F_t)_\delta \subset F_{t + \alpha(\delta)},
\]
and 
\[
\frac{m_G(F_{t-\alpha(\delta)})}{m_G(F_t)} \geq \frac{1}{1 + \beta(\delta)}
\qand
\frac{m_G(F_{t+\alpha(\delta)})}{m_G(F_t)} \leq 1 + \beta(\delta),
\]
Moreover, by Corollary \ref{corPofiniteinFt}, 
\[
M := \sup_t \frac{|P_0 \cap F_t|}{m_G(F_t)} < \infty.
\]
Hence, if we define 
\[
\Psi(t) = \frac{1}{m_G(F_t)} \sum_{x \in P_0 \cap F_t} \psi(x), 
\quad
\Psi_{-} = \varliminf_t \Psi(t) \qand \Psi_{+} = \varlimsup_t \Psi(t),
\]
then it follows that for all $t > 0$
\[
\frac{1}{1 + \beta(\delta)} \Psi(t-\alpha(\delta)) - \eps(1 + \beta(\delta)) M
\quad \leq \quad
\Xi_\rho(t) 
\quad \leq \quad
(1 + \beta(\delta))( \Psi(t) + \eps M),
\]
and thus in particular
\[
\frac{1}{1 + \beta(\delta)} \Psi_{+} - \eps (1 + \beta(\delta)) M 
\quad \leq \quad
\eta_{\nu}(\rho^* * f)
\quad \leq \quad  (1 + \beta(\delta)) (\Psi_{-} + \eps M).
\]
Note that these estimates are uniform in $\eps$. We may now choose a decreasing sequence $(\eps_n)$ which converges
to zero, and pick $\delta_n$ and $\rho_n$ correspondingly. Since $f$ has compact support and $\eta_{\nu}$ is finite on 
compact subsets of $G$, we have 
\[
\lim_{n\to \infty} \eta_{\nu}(\rho_n^* * f) = \eta_{\nu}(f),
\]
and thus, since $\beta$ is continuous and $\beta(0) = 0$, we have 
\[
\Psi_{+} \leq \eta_{\nu}(f) \leq \Psi_{-}.
\]
This shows that $\Psi_{+} = \Psi_{-} = \eta_{\nu}(f)$, and thus finishes the proof.
\end{proof}
We now turn to the proofs of the lemmas.
\begin{proof}[Proof of Lemma \ref{approxlemma}]
Let $\psi$ be a left uniformly continuous function on $G$. Fix $\eps > 0$ and choose $\delta > 0$ such that
\begin{equation}
\label{unifcntbnd}
|\psi(s) - \psi(x)| < \eps, \quad \textrm{for all $s,x \in G$ such that $s^{-1}x \in B_\delta$}.
\end{equation}
Let $\rho$ be a non-negative continuous function on $G$ supported on $B_\delta$ with $\int_G \rho \,dm_G = 1$ and $L \subset G$ be a compact set. Note firstly that if $s^{-1}x \in B_\delta$, $s \in L$, $x \in P_0$ and then $x \in P_0 \cap L_\delta$. This implies that 
\begin{eqnarray*}
\int_L \mathcal P \rho(s^{-1} \cdot P_0)  \psi(s) \, dm_G(s)
&=&
\sum_{x \in P_0} \int_L \rho(s^{-1}x) \psi(s) \, dm_G(s) \\
&=& 
\sum_{x \in P_0 \cap L_\delta} \int_L \rho(s^{-1}x) (\psi(s) - \psi(x)) \, dm_G(s) \\
&+&
\sum_{x \in P_0 \cap L_\delta} \Big( \int_L \rho(s^{-1}x) \, dm_G(s) \Big) \, \psi(x).
\end{eqnarray*}
By the relation between $\psi$ and $B_\delta$ described in \eqref{unifcntbnd}, and by the bound 
$\int_L \rho(s^{-1}x) \, dm_G(s) \leq 1$ for all $x \in G$, we see that
\[
\int_L \mathcal P \rho(s^{-1} \cdot P_0)  \psi(s) \, dm_G(s)
\leq 
\eps |P_0 \cap L_\delta| + \sum_{x \in P_0 \cap L_\delta} \psi(x),
\]
which finishes the proof of the upper bound. Concerning the lower bound, we observe that if $x \in L_{-\delta}$, then  $L^{-1}x \supset B_\delta$ and thus
\begin{equation*}
\int_L \rho(s^{-1}x) \, dm_G(s) = 1 \quad \textrm{for all } x \in L_{-\delta}.
\end{equation*}
Combining this with \eqref{unifcntbnd} we conclude that
\begin{eqnarray*}
\int_L \mathcal P \rho(s^{-1} \cdot P_0)  \psi(s) \, dm_G(s)
&\geq & 
\sum_{x \in P_0 \cap L_{\delta}}  \Big( \int_L \rho(s^{-1}x) \, dm_G(s) \Big) \, \psi(x) \\
&-& 
\sum_{x \in P_0 \cap L_\delta} \int_L \rho(s^{-1}x) |\psi(s) - \psi(x)| \, dm_G(s) \\
&\geq &
 \sum_{x \in P_0 \cap L_{-\delta}}  \Big( \int_L \rho(s^{-1}x) \, dm_G(s) \Big) \, \psi(x) - \eps |P_0 \cap L_\delta| \\
&=&
\sum_{x \in P_0 \cap L_{-\delta}}  \psi(x) - \eps |P_0 \cap L_\delta|,
\end{eqnarray*}
which is the desired lower bound.
\end{proof}

\begin{proof}[Proof of Lemma \ref{PofiniteinFt}] Fix $\delta > 0$ and choose a non-negative continuous function $\rho$ on $G$ supported on $B_\delta$ with 
$\int_G \rho \, dm_G = 1$. We recall from Proposition \ref{PropSiegel} that $M_\delta := \|\mathcal P \rho\|_\infty < \infty$. Now let $L \subset G$ be a compact set and note that for all $x \in L_{-\delta}$, we have $L^{-1}x \supset B_\delta$, and 
thus $\int_L \rho(s^{-1}x) \, dm_G(s) = 1$. We conclude that
\begin{eqnarray*}
M_\delta &\geq & \frac{1}{m_G(L)}  
\int_L \mathcal P \rho(s^{-1} P_0) \, dm_G(s) \quad = \quad 
\frac{1}{m_G(L)} \sum_{x \in P_0} \int_L \rho(s^{-1}x) \, dm_G(s) \\
&\geq &
\frac{1}{m_G(L)} \sum_{x \in P_0 \cap L_{-\delta}} \int_L \rho(s^{-1}x) \, dm_G(s)  \quad = \quad 
\frac{|P_0 \cap L_{-\delta}|}{m_G(L)}.\qedhere
\end{eqnarray*}
\end{proof}
This completes the proof of Theorem \ref{approxthm1}.

\appendix

\section{The Chabauty--Fell topology and the local topology}\label{AppendixLocalTopology}

The goal of this appendix is to compare different topologies on the collection $\mathcal C(G)$ of closed subsets of a lcsc group $G$, and to discuss orbit closures of FLC sets in these topologies. All the results presented in this appendix are well-known in the abelian case \cite{Moody-97, Schlottmann-99, BaakeL-04} and the generalizations to non-abelian groups discussed here are entirely routine. In the non-abelian case, the only treatment we are aware of is \cite{Yokonuma2005}, which however focuses on different aspects. 

\subsection{A basis for the Chabauty--Fell topology of a lcsc group}
We start by discussing different bases for the \emph{Chabauty--Fell topology} on $\mathcal C(G)$. In Subsection \ref{SubsecHull} we have defined this topology by means of the basic open subsets $U_V, U^K \subset \mathcal C(G)$ given by
\[
U_V = \{C \in \mathcal C(G)\mid C \cap V \neq \emptyset\} \quad \text{and} \quad U^K = \{C \in \mathcal C(G)\mid C \cap K = \emptyset\},
\] 
where $V$ runs over the open subsets of $G$ and $K$ runs over the compact subsets of $G$. In the sequel will prefer to work with a different basis, which is defined as follows.

Given $P \in \mathcal C(G)$, $K \in \mathcal K(G)$ and $V \in \mathfrak U(G)$ we define
\[
\widehat{U}_{K,V}(P) := \{Q \in \mathcal C(G)\mid Q \cap K \subset VP \text{ and } P \cap K \subset VQ\}.
\]
\begin{proposition}\label{LMT} The sets $\{\widehat{U}_{K,V}(P)\mid{K \in \mathcal K(G), V \in \mathfrak U(G)}\}$ generate the neighbourhood filter of $P$ in the Chabauty--Fell topology.
\end{proposition}
\begin{proof} Denote by $\tau$ the topology with neighbourhood filters given by the $\{\widehat{U}_{K,V}(P)\}$. We first show that every non-empty Chabauty--Fell open set $U$ contains a non-empty $\tau$-open subset. We may assume that $U$ is of the form 
\[
U = U^K \cap \bigcap_{i=1}^n U_{V_i},
\]
$V_i \in \mathcal O(G)$ for all $i \in \{1, \dots, n\}$ and $K \in \mathcal K(G)$. Since $U \neq \emptyset$ we have $V_i \setminus K \neq \emptyset$ for all $i \in I$ and hence we find $x_1,\dots, x_n \in G$ and $W=W^{-1} \in \mathfrak U(G)$ such that 
\[
\overline{W}x_i \subset V_i \setminus K.
\]
Let $P:= \{x_1, \dots, x_n\}$ and $K' := K \cup\overline{W}P$ and note that the latter union is disjoint. We claim that $\widehat{U}_{K', W}(P) \subset U$. Indeed, let $Q \in \widehat{U}_{K', W}(P)$ so that
\[Q \cap K' \subset WP\quad \text{ and } \quad P \cap K' \subset WQ.\]
Then $Q \cap K' \subset \overline{W}P  = K' \setminus K$ and thus $Q \cap K = \emptyset$, i.e. $Q \in U^K$. Given $i \in \{1, \dots, n\}$ we have $x_i \in P =P \cap K' \subset WQ$ and hence $Wx_i \cap Q = W^{-1}x_i \cap Q \neq \emptyset$. We deduce that $Q \cap V_i \supset Q \cap Wx_i \neq \emptyset$ and thus $Q \in U_{V_i}$. This shows that $Q \in U$ and shows that $\tau$ is finer than the Chabauty--Fell topology.

Conversely let $P \in \mathcal C(G)$, $K \in \mathcal K(G)$ and $V \in \mathfrak U(G)$. We construct a Chabauty--Fell open subset $U$ of $\widehat{U}_{K, V}(P)$ as follows. Firstly, let $W \in \mathfrak U(G)$ be open and symmetric with $W^2 \subset V$. Secondly, let $K' := K\setminus WP \in \mathcal K(G)$. Since $K \cap P$ is compact there exist $t_1, \dots, t_n \in G$ such that
\begin{equation}\label{KQLMT}
K \cap P \subset \bigcup_{i=1}^n Wt_i \quad \text{and} \quad (K \cap P) \cap Wt_i \neq \emptyset.
\end{equation}
We claim that \[U := U^{K'} \cap \bigcap U_{Wt_i} \subset \widehat{U}_{K, V}(P).\] Indeed, let $Q \in U$. Since $Q \in U^{K'}$ we have $\emptyset = Q \cap K' = Q \cap (K\setminus WP)$, hence $Q \cap K \subset WP \subset VP$. Concerning the other inclusion, note that for $i=1, \dots, n$ we have $Q \cap Wt_i \neq \emptyset$, say $q = wt_i$ with $q \in Q$ and $w \in W$. Then $Wt_i = Ww^{-1}q \in WW^{-1}Q \subset VQ$ and thus $K \cap P \subset VQ$ by \eqref{KQLMT}, finishing the proof.
\end{proof}

In the abelian context, this model for the Chabauty--Fell topology appears in \cite{BaakeL-04}, where it is referred to as the \emph{local rubber topology}. 

\subsection{The local uniformity and the local topology} In this subsection we are going to define a $G$-invariant uniformity on $\mathcal C(G)$ whose associated topology is finer than the Chabauty--Fell topology, but coincides with the Chabauty--Fell topology on the orbit closure of any set of finite local complexity. For every $K\in \mathcal K(G)$ and every $V \in\mathfrak U(G)$ we define a subset $U_{K,V} \subset \mathcal C(G) \times \mathcal C(G)$ by
\begin{equation}\label{UKV}
U_{K,V} := \{(P, Q) \in \mathcal{C}(G)\mid \exists t \in V:\; P \cap K = tQ \cap K\}.
\end{equation}
\begin{proposition}\label{LocalUniformity}
The set $\mathcal B := \{U_{K, V}\mid K\in \mathcal K(G), V \in\mathfrak U(G)\}$ is a fundamental system of entourages for a uniformity on $\mathcal{C}(G)$.
\end{proposition}
\begin{proof}In the notation of \cite[Chapter 2, {\S} 1.1]{Bourbaki} we have to show that
\begin{enumerate}[(B1)]
\item the diagonal  $\Delta(\mathcal{C}(G))$ is contained in every $U \in \mathcal B$;
\item for all $U_1, U_2 \in \mathcal B$ there exists $U_3 \in \mathcal B$ with $U_3 \subset U_1 \cap U_2$;
\item for all $U_1 \in \mathcal B$ there exists $U_2 \in \mathcal B$ such that $U_2 \subset U_1^{-1}$;
\item for all $U_1 \in \mathcal B$ there exists $U_2 \in \mathcal B$ such that $U_2^2 \subset U_1$.
\end{enumerate}
We establish (B1) -- (B4) for our $\mathcal B$ at hand.

(B1) is immediate from the fact that $e \in V$ for every $V \in \mathfrak U(G)$. 

(B2) If $(P,Q) \in U_{K_1 \cup K_2, V_1 \cap V_2}$, then there exists $t \in V_1 \cap V_2$ such that
\[
P \cap  (K_1 \cup K_2)= tQ \cap(K_1 \cup K_2). 
\]
It follows that for $j=1,2$,
\[
P \cap K_j = tQ \cap K_j,
\]
whence $(P, Q) \in U_{K_1, V_1} \cap U_{K_2, V_2}$. This shows that $U_{K_1 \cup K_2, V_1 \cap V_2} \subset U_{K_1, V_1} \cap U_{K_2, V_2}$.

(B3) Let $V \in \mathfrak U(G)$ and $K \in \mathcal K(G)$. There exists a compact $W \in \mathfrak U(G)$ with $W^2 \subset V$ and $W = W^{-1}$. Then $WK$ is compact and if $(P, Q) \in U_{WK, W}$, then there exists $t \in W$ such that
\[
P \cap WK = tQ \cap WK.
\]
By assumption, $s := t^{-1} \in W \subset V$ and $e \in sW$, hence $K \subset sWK$. We obtain
\[
Q \cap K = (Q \cap sWK) \cap K = (sP \cap sWK) \cap K =  sP \cap K,
\]
showing that $(Q, P) \in U_{K, V}$, whence $U_{WK, W} \subset U_{K,V}^{-1}$.

(B4) Let $V, K, W$ as in the proof of (B3) and let $(P,R)\in U_{WK, W}^2$. Then there exist $Q \in \mathcal{D}(G)$ such that $\{(P, Q), (Q, R)\} \subset U_{WK, W}$, i.e. there exist
$t_1, t_2 \in W$ such that
\[
P \cap WK = t_1Q \cap WK, \quad Q \cap WK = t_2R \cap WK.
\]
Let $t := t_1t_2 \in V$. Since $K \subset t_1WK \cap WK$ we then have
\begin{eqnarray*}
tR \cap K &=& (t_1t_2 R \cap t_1WK) \cap K = t_1(t_2R \cap WK)\cap K = t_1(Q \cap WK) \cap K\\
&=& t_1Q \cap K = (t_1Q \cap WK)  \cap K = (P \cap WK) \cap K = P \cap K,
\end{eqnarray*}
hence $(P, R) \in U_{V, K}$, showing that $U_{WK, W}^2\subset U_{K, V}$. 

This establishes (B1) -- (B4) and finishes the proof.
\end{proof}
In the sequel we refer to the uniformity defined in Proposition \ref{LocalUniformity} as the 
\emph{local uniformity} on $\mathcal{C}(G)$. and the corresponding topology as the \emph{local topology}. 
By definition, a neighbourhood basis of $P \in \mathcal C(G)$ in the 
local topology is given by the sets
\begin{eqnarray*}
U_{K,V}(P) &=& \{Q \in \mathcal C(G)\mid (P, Q) \in U_{K,V}\}\\
&=& \{Q \in \mathcal C(G) \mid \exists\, t \in V:\; tQ\cap K = P \cap K\}.
\end{eqnarray*}
where $K$ runs through $\mathcal K(G)$ and $V$ runs through $\mathfrak U(G)$
\begin{lemma}\label{LocalTopologyContinuity} The $G$-action on $\mathcal C(G)$ by left-translations is jointly continuous with respect to the local topology.
\end{lemma}
\begin{proof} Let us denote by $m: G \times \mathcal{C}(G) \to  \mathcal{C}(G)$ the left-translation action of $G$, and let $g \in G$, $P \in \mathcal{C}(G)$. We are going to show continuity of $m$ at $(g,P)$. For this let $K\subset G$ be compact and $V \subset G$ be an open identity neighbourhood. We choose a symmetric identity neighbourhood $W$  with $W^2 \subset V$ and define $K' := g^{-1}K$ and $V' := g^{-1}Wg$. Now let $h \in Wg$ and $Q \in U_{K', V'}(P)$. We then find $s \in V'$ such that
$sQ\cap K' = P \cap K'$ and thus
\[
 (gsh^{-1})(hQ) \cap K = gP \cap K.
\]
Since $s \in V'$ we have $gsg^{-1} \in W$ and since also $gh^{-1} \in W$ we obtain 
\[t:=gsh^{-1} = (gsg^{-1})(gh^{-1}) \in W^2 \subset V.\]
To summarize, we have found $t \in V$ such that
\[
t(hQ)\cap K = gP \cap K.
\]
This shows that $hQ \in U_{K, V}(gP)$ and thus $Wg \times U_{K', V'}(P) \subset m^{-1}(U_{K, V}(gP))$, which implies continuity of $m$ at $(g,P)$.
\end{proof}
For the following proposition we denote by $\tau_{CF}$ the Chabauty--Fell topology on $\mathcal C(G)$ and by $\tau_{\rm loc}$ the local topology.
\begin{proposition} The identity map $(\mathcal C(G) , \tau_{\rm loc}) \to (\mathcal C(G), \tau_{CF})$ is continuous, i.e. the local topology is finer than the Chabauty topology (hence in particular Hausdorff).
\end{proposition}
\begin{proof} We show that every $\tau_{CF}$-neighbourhood $\widehat{U}$ of $P \in \mathcal C(G)$ contains a $\tau_{\rm loc}$-neighbourhood of $P$. By Proposition \ref{LMT} we may assume that $\widehat{U} = \widehat{U}_{K, V}(P)$ for some $K \in \mathcal K(G)$, $V \in \mathfrak U(G)$. Let $K'$ be a compact set such that
\[
K \subset \bigcap_{t \in V} t^{-1}K' 
\]
and let $V' \in \mathfrak U(G)$ such that $(V')^{-1}V' \subset V$. We claim that $U_{K', V'}(P) \subset \widehat{U}_{K, V}(P)$. Indeed, let $Q \in U_{K', V'}(P)$ and let $t \in V'$ such that $tQ \cap K' = P \cap K'$. Then $P \cap K \subset V'Q' \subset VQ$ and moreover
\[
Q \cap t^{-1}K' = t^{-1}P \cap t^{-1}K' \Rightarrow Q \cap K = t^{-1}P \cap K \subset (V')^{-1}P \cap K \subset VP \cap K.
\]
This shows that $Q \in \widehat{U}_{K, V}(P)$ and finishes the proof.
\end{proof}
\subsection{Compactness in the local topology}
In the sequel we denote by $\mathcal D(G) \subset \mathcal C(G)$ the subset of locally finite subsets of $G$.
\begin{proposition} The restriction of the local uniformity to $\mathcal D(G)$ is complete.
\end{proposition}
\begin{proof} Let $(I, \leq)$ be a directed set and $(P_i)_{i \in I}$ be a Cauchy net in $\mathcal{D}(G)$ with respect to the local uniformity. We have to show that $(P_i)_{i \in I}$ admits a convergent subnet. Either there exists a subnet converging to the empty set, or after passing to a subnet we may assume that there exists $K \in \mathcal K(G)$ such that $P_i \cap K \neq \emptyset$ for all $i \in I$. In the latter case we can choose $t_i \in P_i \cap K$ and passing to another subnet we may assume that $t_i \to t$. Then $(t_i^{-1}P_i)$ is again a Cauchy net, and convergence of $(P_i)$ is equivalent to convergence of $(t_i^{-1}P_i)$. We may thus assume that $e \in P_i$ for every $i \in I$. 

Now let $V_0$ be a compact identity neighbourhood. Since $(P_i)$ is a Cauchy net there exists $i_0 \in I$ such that for every $i \geq i_0$ we have $(P_{i_0}, P_i) \in U_{V_0, V_0}$. Thus  for every $i \geq i_0$ there exists $s_i \in V_0$ such that
\[
s_iP_{i_0} \cap V_0 = P_i \cap V_0.
\]
Since $e \in P_i \cap V_0 \subset s_iP_{i_0}$ we have $s_i \in P_{i_0}^{-1}$. Since also  $s_i \in V_0$, we deduce that for all $i \geq i_0$,
\begin{equation}\label{CompletenessMain}
P_i \cap V_0 = s_iP_{i_0} \cap V_0 = s_i(P_{i_0} \cap s_i^{-1}V_0) \subset (P_{i_0}^{-1} \cap V_0)(P_{i_0} \cap V_0^{-1}V_0).
\end{equation}
Now $P_{i_0}$ and hence $P_{i_0}^{-1}$ are locally finite, and $V_0$ and hence $V_0^{-1}V_0$ are compact. It follows that
\[
\mathcal F:= (P_{i_0}^{-1} \cap V_0)(P_{i_0} \cap V_0^{-1}V_0)
\]
is finite, and hence $V_1 := (V_0 \setminus \mathcal F) \cup\{e\}$ is an identity neighbourhood. By \eqref{CompletenessMain} we have
\[
P_i \cap V_1 = \{e\}
\]
for all $i \geq i_0$. 

Now let $K \in \mathcal K(G)$ be arbitrary. Since $(P_i)$ is a Cauchy filter we find $i_1 \geq i_0$ such that for all $l, m \geq i_1$ we have $(P_l, P_m) \in U_{K, V_1}$. Thus there exists $t \in V_1$ such that
\[
tP_l \cap K = P_m \cap K.
\]
Now assume $K \supset V_1$. Since $e \in P_l$ we have $t \in tP_l \cap K$, and hence $t \in P_m \cap K$. In particular, $t \in P_m \cap V_1 = \{e\}$ and thus $P_l \cap K = P_m \cap K$. To summarize, for every sufficiently large $K$ (and hence, a posteriori, for every $K$), there exists $i_K \in I$ such that for all $j \geq i_K$ we have $P_j \cap K = P_{i_K} \cap K$. We may assume that $i_K \leq i_{K'}$ whenever $K \subset K'$. In particular, if we define
\[
P := \bigcup_{K \in \mathcal K(G)} P_{i_K} \cap K
\]
then for all $j \geq i_K$ we have
\[
P_{j} \cap K = P \cap K.
\]
This shows both that $P$ is locally finite, since $P \cap K$ is finite for every $K$, and that $(P_i)$ converges to $P$.
\end{proof}
From this completeness property we derive the following compactness criterion.
\begin{corollary}\label{CompactnessCriterion} A subset $C \subset \mathcal D(G)$ is pre-compact with respect to the local topology if and only if for every $K \in \mathcal K(G)$ and $U \in \mathfrak U(G)$ there exists a finite subset $\mathcal F \subset \mathcal D(G)$ such that
\[
C \subset \bigcup_{P \in \mathcal F } U^{-1}_{K, V}(P).
\]
\end{corollary}
\begin{proof} Since the local topology on $\mathcal D(G)$ is complete, it follows from  \cite[Chapter 2, {\S}{} 4.2, Theorem 3]{Bourbaki} that a subset $C \subset \mathcal D(G)$ is pre-compact if and only if for every $K \in \mathcal K(G)$ and $V \in \mathfrak U(G)$ there exists a finite subset $\mathcal F \subset \mathcal D(G)$ such that
\[
C \subset \bigcup_{P \in \mathcal F} U_{K, V}(P).
\]
By axiom (B3) in the proof of Proposition \ref{LocalUniformity} we can replace $U_{K,V}$ by $U_{K, V}^{-1}$. 
\end{proof}

\subsection{Orbit closures of sets of finite local complexity} It turns out that finite local complexity can be characterized as a compactness property as follows.
\begin{theorem}\label{FLC} Let $P \in \mathcal D(G)$. Then the following are equivalent.
\begin{enumerate}[(i)]
\item $P$ has finite local complexity, i.e. $P^{-1}P \subset G$ is locally finite.
\item The $G$-orbit $G.P \subset \mathcal C(G)$ is pre-compact with respect to the local topology.
\item $\forall K \in \mathcal K(G)\; \exists K' \in \mathcal K'(G)\; \forall t \in G\; \exists t' \in K': \; tP\cap K = t'P \cap K.$
\end{enumerate}
\end{theorem}
\begin{proof} (i) $\Rightarrow$ (ii): Assume that $P^{-1}P$ is locally finite and fix $K \in \mathcal K(G)$ and $V \in \mathfrak U(G)$. Then the intersection $F:= P^{-1}P \cap K^{-1}K$ is finite since $K^{-1}K$ is compact. Moreover, finitely many right-$V$-translates cover $K$, i.e. there exists another finite set $E$ such that $K \subset VE$. 
We claim that for all $g \in G$ we get
\begin{equation}\label{ClaimCompactness}
gP  \in \bigcup_{P' \in \mathcal F} U_{K,V}^{-1}(P'),
\end{equation}
where
\[
\mathcal F := \{sF' \in \mathcal D(G)\mid F' \subseteq F, s \in E\}.
\] 
Since $\mathcal F \subset \mathcal D(G)$ is finite, this will imply pre-compactness of $G.P$ by Corollary \ref{CompactnessCriterion}. Thus its remains only to show \eqref{ClaimCompactness}.

If  $gP \cap K = \emptyset$ then there is nothing to show. Otherwise we can choose $p \in P$ such that $gp \in K \subset VE$. Then $F':= p^{-1}P \cap (gp)^{-1}K \subseteq F$ and we find
$s \in E$ and $v \in V$ such that $gp = vs$. We then compute
\[
gP\cap K = gp(p^{-1}P \cap (gp)^{-1}K)\cap K = vs(p^{-1}P \cap (gp)^{-1}K)\cap K = vsF' \cap K, 
\]
which shows that $gP \in U_{K,V}^{-1}(sF')$ and proves \eqref{ClaimCompactness}.

(ii) $\Rightarrow$ (iii): If the orbit $G.P$ is precompact, then for every $K \in \mathcal K(G)$ and $V \in \mathfrak U(G)$ the open cover
\[
G.P \subset \bigcup_{g \in G} U^{-1}_{K, V}(g.P)
\]
has a finite subcover. Given $K \in \mathcal K(G)$ we can thus choose a compact $V \in \mathfrak U(G)$ and $t_1, \dots, t_n \in G$ such that
\[
G.P \subset \bigcup_{i=1}^n U^{-1}_{K, V}(t_iP).
\]
Set $K' := \bigcup Vt_i$, then for every $t \in G$ there exists $s \in V$ and $i \in \{1, \dots, n\}$ such that
\[
tP \cap K = st_iP \cap K.
\]
Hence if we define $t' := st_i$, then $t\in K'$ and $tP \cap K=t'P \cap K$.

(iii) $\Rightarrow$ (i): Given $P$ satisfying (iii) we will show that $P^{-1}P \cap K$ is finite for every $K \in \mathcal K(G)$. We may assume that $e \in K$ and choose $K' \in \mathcal K(G)$ as in (iii). We will show that
\begin{equation}\label{FLCExplicit}
P^{-1}P \cap K \subset (P^{-1} \cap K')(P \cap (K')^{-1}K),
\end{equation}
which is finite. Thus let $q \in P^{-1}P \cap K$ and choose $q_1, q_2 \in P$ with $q= q_2^{-1}q_1$. By assumption there exists $t' \in K'$ such that
\[
q_2^{-1}P \cap K = t'P \cap K\subset t'P.
\]
We have $e \in q_2^{-1}P \cap K$, hence $e \in t'P$, i.e. $t' \in P^{-1} \cap K'$. Thus
\[
q = q_2^{-1}q_1 \in q_2^{-1}P \cap K =  t'P \cap K= t'(P \cap (t')^{-1}K)  \subset (P^{-1} \cap K')(P \cap (K')^{-1}K),
\]
which establishes \eqref{FLCExplicit} and finishes the proof.
\end{proof}
\begin{corollary}\label{TopologiesCoincide} Let $P \in \mathcal D(G)$ and ${X} := \overline{G.P} \subset \mathcal C(G)$ its orbit closure in the local topology. Then the Chabauty--Fell topology and the local topology coincide on ${X}$ if and only if $P$ is of finite local complexity. In this case, $X$ is also the orbit closure of $P$ in the Chabauty--Fell topology.
\end{corollary}
\begin{proof} If $P$ has finite local complexity, then $({X}, \tau_{\rm loc})$ is compact and hence the continuous map $({X}, \tau_{\rm loc}) \to ({X}, \tau_{CF})$ is a homeomorphism. This implies in particular, that $X$ is also the orbit closure of $P$ in the Chabauty--Fell topology. Conversely, if the topologies coincide, then $({X}, \tau_{\rm loc})$ is compact, hence $P$ has finite local complexity.
\end{proof}

\bibliographystyle{abbrv}

\end{document}